\documentclass[11pt]{amsart}
\usepackage{amssymb,amsmath,amsthm,enumitem,colonequals,mlmodern,tikz-cd,microtype}
\usepackage[cal=euler,bfcal,bb=px,bfbb]{mathalpha}
\usepackage{epsfig}
\usetikzlibrary{positioning}
\usepackage{url}
\usepackage[colorlinks=true,citecolor=blue,urlcolor=blue,linkcolor=blue]{hyperref}
\usepackage{setspace}
\usepackage{fancyhdr}
\usepackage{pdfpages}
\usepackage{color}
\usepackage{verbatim}
\usepackage{fancyhdr}
\usepackage[top=2.5cm, bottom=2.5cm, left=2cm, right=2cm]{geometry}
\usepackage{xcolor}
\colorlet{darkblue}{blue!55!black}
\colorlet{darkcyan}{cyan!50!black}
\colorlet{darkgreen}{green!60!black} 
\PassOptionsToPackage{hyphens}{url}
\usepackage{hyperref}
\hypersetup{
    colorlinks=true,
    linkcolor=darkblue,
    urlcolor=darkcyan,
    citecolor=darkgreen,
}
\def\eqref#1{\textcolor{darkblue}{(\ref{#1})}}
\usepackage[nameinlink]{cleveref} 
\Crefformat{section}{#2\S#1#3}
\Crefmultiformat{section}{#2\S\S#1#3}{ and~#2#1#3}{, #2#1#3}{, and~#2#1#3}
\usepackage[pagewise]{lineno}
\let\oldequation\equation
\let\oldendequation\endequation
\renewenvironment{equation}{\linenomathNonumbers\oldequation}{\oldendequation\endlinenomath}
\expandafter\let\expandafter\oldequationstar\csname equation*\endcsname
\expandafter\let\expandafter\oldendequationstar\csname endequation*\endcsname
\renewenvironment{equation*}{\linenomathNonumbers\oldequationstar}{\oldendequationstar\endlinenomath}
\let\oldalign\align
\let\oldendalign\endalign
\renewenvironment{align}{\linenomathNonumbers\oldalign}{\oldendalign\endlinenomath}
\expandafter\let\expandafter\oldalignstar\csname align*\endcsname
\expandafter\let\expandafter\oldendalignstar\csname endalign*\endcsname
\renewenvironment{align*}{\linenomathNonumbers\oldalignstar}{\oldendalignstar\endlinenomath}

\newcommand {\Hom} {\mathsf{Hom}}

\newcommand  {\Bun}   {\mathrm{Bun}}

\theoremstyle{plain}
\newtheorem{theorem}{Theorem}[section]

\newtheorem{corollary}[theorem]{Corollary}
\newtheorem{proposition}[theorem]{Proposition}
\theoremstyle{definition}
\newtheorem{definition}[theorem]{Definition}
\newtheorem{example}[theorem]{Example}
\newtheorem{remark}[theorem]{Remark}

\newtheorem{construction}[theorem]{Construction}

\newtheorem*{ack}{Acknowledgments}
\AddToHook{env/conjecture/begin}{\crefalias{theorem}{conjecture}}
\AddToHook{env/lemma/begin}{\crefalias{theorem}{lemma}}
\AddToHook{env/corollary/begin}{\crefalias{theorem}{corollary}}
\AddToHook{env/proposition/begin}{\crefalias{theorem}{proposition}}
\AddToHook{env/definition/begin}{\crefalias{theorem}{definition}}
\AddToHook{env/remark/begin}{\crefalias{theorem}{remark}}
\AddToHook{env/example/begin}{\crefalias{theorem}{example}}
\AddToHook{env/hypothesis/begin}{\crefalias{theorem}{hypothesis}}
\AddToHook{env/notation/begin}{\crefalias{theorem}{notation}}
\numberwithin{equation}{section}
\numberwithin{theorem}{section}

\title[Geometry of a semistable degeneration of moduli of vector bundles on curves]{Geometry of a semistable degeneration of moduli of vector bundles on curves}
\author{Sourav Das}
\begin{document}
\maketitle
\begin{abstract}
We study the Gieseker moduli space $\mathcal{M}_{X_0}$ of stable rank-$2$, degree-$1$ vector
bundles on an irreducible curve $X_0$ with one node. Its normalization is a two-step blow-up of
the moduli space of generalized parabolic bundles on the normalization $\tilde X_0$, and
$\mathcal{M}_{X_0}$ is recovered by gluing along a natural involution on the boundary. From this
we compute $\operatorname{Pic}(\mathcal{M}_{X_0})$ and, via a Harder--Narasimhan stratification,
the virtual Poincar\'e polynomial of $\mathcal{M}_{X_0}$. Since the normalization morphism is 
small, the intersection cohomology of $\mathcal{M}_{X_0}$ equals the ordinary cohomology of the
smooth normalization, which gives a closed formula in terms of the arithmetic genus.
\end{abstract}

\tableofcontents

\section{Introduction}

The moduli spaces of holomorphic vector bundles and Higgs bundles on compact Riemann surfaces (or, equivalently, on smooth projective algebraic curves) form one of the most important and well-studied classes of moduli spaces in algebraic geometry. These spaces lie at the crossroads of several major themes in modern mathematics: they provide concrete geometric realizations of character varieties, serve as the setting for the non-abelian Hodge correspondence, and play a central role in the geometric Langlands program. When the underlying curve is smooth, the geometry of these moduli spaces is relatively well understood. For classical groups, their cohomology rings, Picard groups, and (in the Higgs case) the structure of the Hitchin fibration have been computed in many cases, and deep connections with mirror symmetry and integrable systems have been established.

A fundamental and largely open direction is to understand how these moduli spaces behave when the underlying curve degenerates. In particular, one would like to understand the geometry of moduli spaces of bundles and Higgs bundles on singular curves, and more generally, how these spaces vary in families when the curve acquires nodes or other singularities. Such questions are not only of intrinsic geometric interest but are also motivated by arithmetic geometry, degeneration techniques in mirror symmetry, and the desire to extend the geometric Langlands correspondence to more general settings.

In this paper we begin a systematic investigation of this problem by studying the semistable degeneration of the moduli space of vector bundles when the curve becomes nodal. Concretely, we consider an irreducible projective nodal curve \(X_0\) with exactly one node and study the Gieseker moduli space \(\mathcal{M}_{X_0} = \mathcal{G}_{X_0}(2,1)\) of stable rank-2, degree-1 Gieseker vector bundles on \(X_0\). This space arises naturally as a compactification of the moduli space of stable vector bundles on \(X_0\) and can be viewed as a semistable degeneration of the usual moduli space of vector bundles on smooth curves. 

Our approach is based on the observation, developed in the ICTP notes on H-bundles and generalised parabolic bundles and degenerations, that the Gieseker moduli space admits a natural description in terms of generalized parabolic bundles (GPBs) on the normalization \(\tilde{X}_0\) of the nodal curve. This perspective reduces many questions about the singular moduli space to questions on smooth varieties: the normalization of \(\mathcal{M}_{X_0}\) is realized concretely as an explicit two-step blow-up of the moduli spaces of GPBs, and the singular space is recovered by gluing the resulting boundary divisor to itself along a natural involution. This reduction allows us to bring powerful tools from algebraic geometry to bear on \(\mathcal{M}_{X_0}\), including intersection theory on blow-ups, Mayer--Vietoris type descent for line bundles, Harder--Narasimhan stratifications.

Using this framework we obtain three main results. First, we compute the Picard group of \(\mathcal{M}_{X_0}\) by determining exactly which line bundles on the smooth normalization are invariant under the gluing involution, and therefore descend to the singular space. Second, we compute the virtual Poincaré polynomial of the moduli stack/space of Gieseker bundles, combining a Whitney stratification with a Harder--Narasimhan stratification on the associated stack/space of torsion-free sheaves. Third since the normalization morphism \(\pi: Y \to \mathcal{M}_{X_0}\) is small in the sense of Goresky--MacPherson; the Decomposition Theorem collapses to its simplest possible case, with the pushforward of the constant sheaf on \(Y\) computing the intersection complex of \(\mathcal{M}_{X_0}\) exactly, with no correction terms supported on the singular strata. This gives a closed-form expression for the intersection Poincaré polynomial of \(\mathcal{M}_{X_0}\) directly in terms of the arithmetic genus of \(X_0\).

The main results of this paper are the following three theorems. Throughout, \(g := g(X_0)\) denotes the arithmetic genus of \(X_0\), so that the normalization \(\tilde{X}_0\) has genus \(g-1\).

\

\begin{theorem}
The Picard group of $\mathcal{M}_{X_0}$ is the subgroup of $\operatorname{Pic}(\widetilde{\mathcal M_{{\operatorname{GPB}}}})$ generated by the following two types of line bundles:
\[
\theta_{\widetilde{G}}, \text{and}~(L + b(\delta + \varepsilon_1 + \varepsilon_2)),
\]
where $L \in \mathcal Pic^1(\tilde{X}_0)$ and $b \in \mathbb{Z}$ satisfy
\[
t_{x^- - x^+}^*(L) - L = b(\ell_2 - \ell_1).
\]
\end{theorem}

\

For the definitions of the line bundles involved in the statement can be found in the appropriate section.

\begin{theorem}[Virtual Poincaré polynomial]
\label{thm:vpp}
Let $\mathrm{GVB}$ denote the moduli stack of Gieseker vector bundles of rank $2$ and degree $1$ on $X_0$, and write $g:=g(X_0)$. With $\mathbb{P}(A_d;t)=(1+t)^{2(g-1)}$ and $\mathbb{P}(S_d;t)=(1+t)^{2(g-1)}/(t^2-1)$, the virtual Poincaré polynomial of the Gieseker-semistable locus obtained from the Whitney stratification is
\begin{equation}
\mathbb{P}(GVB^{\mathrm{ss}};t)
= \frac{t^{2}(1+t)^{4g-7}\Bigl[(t^{4}+1)(1-t+t^{2})^{2(g-1)} - 2t^{2g}\Bigr]}{(t-1)^{3}}.
\label{eq:gvbss-v2}
\end{equation}
and
\begin{equation}
\mathbb{P}(\mathcal{M}_{X_0},t)
=(t^2-1)\,\mathbb{P}(GVB^{\mathrm{ss}};t)
= \frac{t^{2}(1+t)^{4g-6}\Bigl[(t^{4}+1)(1-t+t^{2})^{2(g-1)} - 2t^{2g}\Bigr]}{(t-1)^{2}}.
\label{eq:coarse-v2}
\end{equation}
\end{theorem}

\begin{theorem}[Intersection Poincar\'e Polynomial of $\mathcal M_{X_0}$, closed form]
Let $\mathcal M_{X_0}$ be the Gieseker moduli space of stable rank-2, degree-1 bundles on the irreducible nodal curve $X_0$, and write $g := g(X_0)$ for its arithmetic genus, so that $g(\tilde X_0) = g-1$. Then
\[
\resizebox{0.98\linewidth}{!}{$
\displaystyle
P_{IH}(\mathcal M_{X_0}, t)
= \frac{(1+t)^{2(g-1)}\bigl((1+t^3)^{2(g-1)}-t^{2(g-1)}(1+t)^{2(g-1)}\bigr)}{(1-t^2)(1-t^4)}
\times \bigl(1+3t^2+4t^4+3t^6+t^8\bigr)
+ \bigl(t^{2(g-1)}+t^{2g}\bigr)(1+t)^{4(g-1)}
$}
\]
\end{theorem}

\

The paper is organized as follows. In \S2 we prove the Picard group theorem. \S3 is devoted to the virtual Poincaré polynomial, via the Whitney and Harder--Narasimhan stratifications. In \S4 we show that the normalization morphism is small, apply the Decomposition Theorem to obtain the intersection Poincaré polynomial without correction terms, and make the resulting formula fully explicit via GIT-variation.

\subsubsection{Future Projects}

The present paper treats rank $2$, degree $1$ bundles on an irreducible nodal curve with a single node, but the methods are built to generalize. This is the first step in a broader program aimed at the geometry of semistable degenerations of moduli spaces of vector and Higgs bundles. The Mayer--Vietoris approach used here for the Picard group is intended to be adapted to compute the ordinary Poincar\'e polynomial and Hodge numbers of these singular moduli spaces, and these computations are intended to be extended to arbitrary rank and degree for both vector and Higgs bundles, as well as to the corresponding moduli spaces associated with arbitrary semisimple linear algebraic groups. It would be interesting to treat  the compactified moduli of semistable (Higgs) bundles on stable curves with only rational components, where the moduli space depends on a choice of polarization and so, unlike in the irreducible case treated here, is no longer unique for a fixed curve.
\

\begin{ack}
The author thanks S. Basu for helpful discussions and for suggesting the reference to the BDGG decomposition theorem. The computation in \S4 was carried out during an extension of the author's doctoral studies at the Chennai Mathematical Institute. This research was supported in part by the International Center for Theoretical Sciences (ICTS) for participating in the program Monopole Moduli 2026 (code: ICTS/MM2026/09).
\end{ack}

\begin{ack}
The author thanks Claude (Anthropic, Sonnet 5) for assistance over several rounds of interactive debugging in checking the theorem on  virtual Poincaré polynomial computations of \S\ref{thm:vpp}: symbolic cross-checking of specific algebraic identities surfaced an extra residue in the formula for $\mathbb P(GVB^{\mathrm{ss}};t)$, which was traced to an error in the stratification of $HN^2$ (the locus $HN^{2,1}$, of local type $\mathcal O\oplus\mathfrak m$ at the node, had incorrectly been grouped with $HN^{2,2}$ under the $\mathbb P^3$-bundle structure of $\theta$, rather than recognized as a further isomorphism locus); correcting this resolved the inconsistency. This assistance was limited to computational verification and identification of this specific error; it does not constitute an independent certification of the correctness of the paper's mathematical content, responsibility for which remains with the author.
\end{ack}
\

\section{Picard Group}

The Picard group of a moduli space plays a central role in understanding its geometry, intersection theory, and birational properties: it governs which line bundles exist, controls linear systems and morphisms to projective space, and often detects the failure of a compactification to be $\mathbb Q$-factorial or locally factorial at its singular points. In this section we compute the Picard group of Gieseker's moduli space $\mathcal M_{X_0}=\mathcal G_{X_0}(2,1)$ of stable rank-2, degree-1 Gieseker vector bundles on an irreducible nodal curve $X_0$ with exactly one node.

Since $\mathcal M_{X_0}$ is singular, we pass to its normalization $\widetilde{\operatorname{GPB}}_{\tilde X_0}$, a smooth variety obtained from the moduli space of generalized parabolic bundles (GPBs) on the normalization $\tilde X_0$ of $X_0$ by two explicit blow-ups. On this smooth model the divisor class group is computed by ordinary intersection theory, and $\operatorname{Pic}(\mathcal M_{X_0})$ is then recovered as the subgroup of line bundles that descend along the normalization map, concretely, the classes on $\widetilde{\operatorname{GPB}}_{\tilde X_0}$ left invariant by the natural involution that records how the normalization glues back together to reproduce the singularities. 

We describe the generators of $\operatorname{Pic}(\widetilde{\operatorname{GPB}}_{\tilde X_0})$, the pullback of the determinant bundle from $\operatorname{Bun}_{\tilde X_0}(2,1)$, the two boundary divisors $\overline{D_1},\overline{D_2}$ arising from the blow-up construction, the two exceptional divisors $E_1,E_2$, and a handful of tautological classes pulled back from $\tilde X_0$ itself  and determine exactly which combinations of them are invariant under the involution swapping $\overline{D_1}\leftrightarrow\overline{D_2}$ and $E_1\leftrightarrow E_2$. This total boundary divisor $D=\overline{D_1}\cup\overline{D_2}\cup E_1\cup E_2$, together with the symmetry the involution imposes on it, is what makes the descent computation tractable: rather than analyzing $\operatorname{Pic}(\mathcal M_{X_0})$ directly on a singular space, we reduce the problem to linear algebra on $\operatorname{Pic}(\widetilde{\operatorname{GPB}}_{\tilde X_0})$.

\subsubsection{Introduction to Gieseker's moduli space}

For a smooth projective curve, the moduli space of stable vector bundles of fixed rank and degree is already well understood; the difficulty appears the moment the curve degenerates. If $X_0$ is an irreducible curve with a single node, the naive moduli space $U_{X_0}(2,1)$ of stable bundles on $X_0$ itself is not compact: as a family of bundles on $X_0$ degenerates, the limit typically fails to be locally free at the node, and one is forced to leave the category of vector bundles altogether and pass to torsion-free sheaves. Gieseker's insight was that one can instead compactify $U_{X_0}(2,1)$ while staying entirely within the world of \emph{vector bundles}, at the cost of allowing the underlying curve itself to degenerate further  by inserting a chain of rational curves at the node and letting the bundle become non-trivial along that chain. This produces a compactification that is particularly well suited to studying bundles on irreducible nodal curves, since every boundary point still has a transparent modular meaning as an honest vector bundle on some (mildly degenerated) curve.

\subsubsection{The setting}

Let $X_0$ be an irreducible projective curve of arithmetic genus $g=p_a(X_0)$ with exactly one node $p$, and let $\nu:\tilde X_0\to X_0$ be its normalization, a smooth curve of genus $g-1$. Write $\nu^{-1}(p)=\{x^+,x^-\}$ for the two points of $\tilde X_0$ lying over the node. We fix throughout the numerical invariants rank $r=2$ and degree $d=1$; since $\gcd(r,d)=1$, every semistable bundle of this rank and degree is automatically stable, which simplifies the moduli theory considerably  there are no strictly semistable points to worry about, and correspondingly no extra automorphisms to quotient by.

\subsubsection{Gieseker curves (expanded degenerations)}

The idea behind Gieseker's compactification is to give a vector bundle "room to degenerate" by inserting an exceptional chain of rational curves at the node, rather than forcing it to become non-locally-free there. Concretely, for each integer $k\geq 0$ the \emph{Gieseker curve} $X_k$ of length $k$ is obtained from $\tilde X_0$ by attaching a chain of $k$ copies of $\mathbb{P}^1$ between the two points $x^+$ and $x^-$:
\[
X_{k} = \tilde{X}_0 \cup R_1 \cup \cdots \cup R_k,
\]
where consecutive components meet transversally at a single point, $\tilde{X}_0 \cap R_1 = \{x^+\}$, and $R_k \cap \tilde{X}_0 = \{x^-\}$ (so for $k=0$ we simply glue $x^+$ to $x^-$ and recover $X_0$ itself). Each $X_k$ comes with a contraction morphism $\pi_k: X_k \to X_0$ that collapses the entire chain $R_1\cup\cdots\cup R_k$ to the single point $p$; this is the map along which a bundle on $X_k$ degenerates back down to a (typically non-locally-free) sheaf on $X_0$. The family $\{X_k\}_{k \geq 0}$, together with these contraction maps, constitutes the expanded degeneration of $X_0$ underlying the whole construction.

\subsubsection{Gieseker (admissible) bundles}

Not every vector bundle on $X_k$ is allowed to appear as a limit: the restriction to each rational bridge $R_i$ must be neither too split (which would make that component of the chain contribute nothing) nor too unbalanced (which would push the bundle out of the stable range). This is made precise as follows. A \emph{Gieseker vector bundle} of rank 2 and degree 1 on \(X_k\) is a vector bundle \(E\) on \(X_k\) satisfying:
\begin{itemize}
    \item \(E|_{\tilde{X}_0}\) is a vector bundle of rank 2 and degree 1 on the smooth curve $\tilde X_0$;
    \item on each rational component \(R_i \cong \mathbb{P}^1\), the restriction \(E|_{R_i}\) is \emph{admissible}: it splits as \(\mathcal{O}(a) \oplus \mathcal{O}(b)\) with \(0\leq a\leq 1\), \(0\leq b \leq 1\), and $a,b$ not both zero  that is, $E|_{R_i}$ is as close to trivial as possible while still genuinely using the component $R_i$;
    \item the pushforward \(\pi_{k*}(E)\) is a torsion-free sheaf of rank 2 and degree 1 on \(X_0\).
\end{itemize}
The admissibility condition is what keeps the moduli problem finite-dimensional and well-behaved: it rules out both the redundant components (where $a=b=0$ and $R_i$ could simply be contracted) and the destabilizing ones (where the splitting type is too unbalanced for the resulting bundle to be stable).

\subsubsection{The Moduli Space \(\mathcal{M}_{X_0} = \mathcal{G}_{X_0}(2,1)\)}

The \emph{Gieseker moduli space} \(\mathcal{M}_{X_0}\) parametrizes isomorphism classes of stable Gieseker vector bundles on all of the expanded curves $X_k$ at once. For rank 2 it turns out that chains of length at most 2 already suffice to accommodate every possible degeneration, so it is enough to let $k$ range over $\{0,1,2\}$:
\begin{itemize}
    \item It is a projective scheme that compactifies the usual moduli space \(U_{X_0}(2,1)\) of stable bundles on \(X_0\).
    \item It has a natural stratification by the length \(k\) of the rational chain needed to represent a given bundle, and this stratification will reappear later as the singularity stratification $\overline{M^0}\supset\overline{M^1}\supset M^2$ of $\mathcal M_{X_0}$ itself.
    \item The boundary  the loci with positive chain length ($k\geq1$)  records successively deeper degenerations, with $\mathcal M_{X_0}$ becoming singular exactly where the chain length jumps.
\end{itemize}

\subsubsection{Connection to generalised Parabolic Bundles and Normalization}

Working directly with chains of rational curves is geometrically intuitive but might be, in some cases, computationally difficult. The key simplification, developed in the ICTP notes on generalised parabolic bundles and degenerations, is that the entire moduli problem for $\mathcal M_{X_0}$ can be re-encoded on the smooth curve $\tilde X_0$ alone, with no rational bridges in sight, at the cost of remembering a small amount of extra linear-algebra data at the two points $x^+,x^-$. This is the passage to \emph{generalized parabolic bundles} (GPBs).

The idea is as follows. A rank-2 vector bundle $E$ on $\tilde X_0$, together with a 2-dimensional subspace $S\subset E_{x^+}\oplus E_{x^-}$, determines a torsion-free sheaf $\mathcal F$ on $X_0$: namely, $\mathcal F$ is the sheaf whose sections near $p$ are exactly those sections of $\nu_*E$ whose values at $x^+$ and $x^-$ combine to lie in $S$. When $S$ meets both summands transversally  equivalently, when $S$ projects isomorphically onto each of $E_{x^+}$ and $E_{x^-}$  this simply glues $E_{x^+}$ to $E_{x^-}$ along an isomorphism, so $\mathcal F$ is again locally free and we recover an ordinary bundle on $X_0$ (the $k=0$ stratum). As $S$ degenerates away from this transverse position, $\mathcal F$ develops exactly the more complicated local structure at $p$ that, on the Gieseker side, is resolved by inserting a rational chain  so the pairs $(E,S)$ already encode, for generic points, the same data that the expanded degenerations $X_k$ encode geometrically. Making this dictionary precise, and understanding the relation between the resulting moduli space of pairs $(E,S)$ with the normalization of $\mathcal M_{X_0}$, is the content of the blow-up construction carried out below; for now we set up the relevant moduli stacks before working with the corresponding moduli spaces.

\begin{definition}[Base moduli stack and space]
Let \(\operatorname{Bun}_{\tilde{X}_0}\) denote the moduli stack of rank-2, degree-1 vector
bundles on the smooth curve \(\tilde{X}_0\), and let \(\mathbb{E}\) be the universal bundle on
\(\operatorname{Bun}_{\tilde{X}_0} \times \tilde{X}_0\). We write $\mathbb E_{x^+}$ and
$\mathbb E_{x^-}$ for the restrictions of $\mathbb E$ along
$\{x^+\}\times\operatorname{Bun}_{\tilde X_0}$ and $\{x^-\}\times\operatorname{Bun}_{\tilde X_0}$
respectively  these are rank-2 vector bundles on $\operatorname{Bun}_{\tilde X_0}$, the fibers
of $\mathbb E$ over the two marked points.

Since $\operatorname{rank}=2$ and $\deg=1$ are coprime, every semistable bundle in this class is
automatically stable, so $\operatorname{Bun}_{\tilde X_0}$ has no strictly semistable points and
its coarse moduli space
\[
\mathcal M_{\operatorname{Bun}} \;=\; \mathcal M_{\tilde X_0}(2,1)
\]
is a fine moduli space: it is a smooth projective variety of dimension $4g-7$
(where $g := g(X_0)$, so $g(\tilde X_0)=g-1$), and $\operatorname{Bun}_{\tilde X_0}\to \mathcal M_{\operatorname{Bun}}$ is a
$\mathbb G_m$-gerbe, the automorphism group of a stable bundle being exactly the scalars
$\mathbb C^*\subset \operatorname{Aut}(E)$. A universal bundle
\[
\mathcal E \;\longrightarrow\; B\times \tilde X_0
\]
exists on $\mathcal M_{\operatorname{Bun}}$ itself (rather than only on the stack), and is unique up to tensoring by the
pullback of a line bundle on $\mathcal M_{\operatorname{Bun}}$; we fix one such choice once and for all. We write
$\mathcal E_{x^+}$ and $\mathcal E_{x^-}$ for its restrictions along
$\{x^+\}\times B$ and $\{x^-\}\times \mathcal M_{\operatorname{Bun}}$, the rank-2 vector bundles on $B$ obtained by evaluating
the universal family at the two marked points, so that $\mathcal E_{x^\pm}$ pulls back to
$\mathbb E_{x^\pm}$ under $\operatorname{Bun}_{\tilde X_0}\to \mathcal M_{\operatorname{Bun}}$.
\end{definition}

\begin{definition}[Moduli stack of generalized parabolic bundles (GPBs)]
Consider the Grassmannian bundle over \(\operatorname{Bun}_{\tilde{X}_0}\)
\[
\mathrm{GPB} := \operatorname{Gr}_2\bigl( \mathbb{E}_{x^+} \oplus \mathbb{E}_{x^-} \bigr) \longrightarrow \operatorname{Bun}_{\tilde{X}_0},
\]
whose fiber over a bundle $E$ is the Grassmannian of 2-dimensional subspaces $S\subset E_{x^+}\oplus E_{x^-}$. A point of \(\mathrm{GPB}\) is thus a pair $(E,S)$: a rank 2, degree 1 vector bundle on \(\tilde{X}_0\) together with a \emph{generalized parabolic structure} $S$ at the two marked points \(x^+\) and \(x^-\), in the sense described above.
\end{definition}

\begin{definition}[$\alpha$-(semi)stability of GPBs]\label{def:alpha-stab}
Let $(E,S)$ be a generalized parabolic bundle on $\tilde X_0$, i.e.\ a rank-2, degree-1 vector
bundle $E$ on $\tilde X_0$ together with a $2$-dimensional subspace
$S\subset E_{x^+}\oplus E_{x^-}$. For a nonzero proper subbundle $F\subsetneq E$, write
\[
S_F \;:=\; S\cap \bigl(F_{x^+}\oplus F_{x^-}\bigr)\ \subset\ F_{x^+}\oplus F_{x^-},
\]
and for a parameter $\alpha\in(0,1]$ define the \emph{$\alpha$-degree} of $F$ (relative to $S$) by
\[
\deg{}_\alpha(F,S) \;:=\; \deg(F) + \alpha\,\dim(S_F).
\]
The pair $(E,S)$ is \emph{$\alpha$-semistable} if for every nonzero proper subbundle
$F\subsetneq E$,
\[
\frac{\deg{}_\alpha(F,S)}{\operatorname{rank}(F)} \;\le\; \frac{\deg{}_\alpha(E,S)}{\operatorname{rank}(E)}
\;=\; \frac{\deg(E)+2\alpha}{2},
\]
and \emph{$\alpha$-stable} if the inequality is strict for every such $F$. Since $F$ has rank
$1$ and $\dim S_F\in\{0,1,2\}$, the right-hand side takes only finitely many values as
$\alpha$ ranges over $(0,1]$, so $(0,1]$ decomposes into finitely many open \emph{chambers}
separated by finitely many \emph{walls}, across which the (semi)stable locus  and hence the
moduli space $\mathcal M^\alpha_{\operatorname{GPB}}$  can change.
\end{definition}

\begin{remark}
At $\alpha\to 0^+$ the stability condition degenerates to stability of the underlying bundle
$E$ alone, with $S$ unconstrained; this is the chamber realized by the Grassmannian bundle
$\operatorname{GPB}_{\tilde X_0}=\operatorname{Gr}_2(\mathbb E_{x^+}\oplus\mathbb E_{x^-})$ (equivalently $\mathcal M^{\alpha=0}_{\operatorname{GPB}}$ in the
notation of \S4). As $\alpha$ increases across the walls in $(0,1]$, the $\alpha$-stable
locus is modified along loci where some line subbundle $F\subset E$ achieves equality above;
these are exactly the loci blown up (or flipped) in passing from
$\mathcal M^{\alpha=0}_{\operatorname{GPB}}$ to $\mathcal M^{\alpha=1}_{\operatorname{GPB}}$.
\end{remark}

\begin{definition}[$1$-stability and $\mathcal M^{\alpha=1}_{\operatorname{GPB}}$]\label{def:1-stab}
Specializing \cref{def:alpha-stab} to $\alpha=1$, a pair $(E,S)$ is \emph{$1$-semistable}
(resp.\ \emph{$1$-stable}) if for every line subbundle $F\subset E$,
\[
\deg(F)+\dim(S_F) \;\le\; \tfrac12\bigl(\deg(E)+2\bigr)
\qquad(\text{resp.\ strict}).
\]
This is precisely the stability notion under which the dictionary $(E,S)\mapsto \mathcal F$ identifies $1$-stable GPBs with stable torsion-free sheaves on $X_0$. We write
$\mathcal M^{\alpha=1}_{\operatorname{GPB}}$ for the resulting moduli space of $1$-stable GPBs;
by construction it is the space appearing at the base of the right-hand column of the
flip-flop diagram, mapping to $\mathcal M_{\mathrm{TF}_{X_0}}$.
\end{definition}

\smallskip

    \textbf{For convenience, we will describe the boundary divisor and the natural involution on the normalisation of the boundary divisor on the stacks but these descriptions hold equally for the corresponding moduli schemes.}

\smallskip

This stack \(\mathrm{GPB}\) carries two distinguished boundary divisors \(D_1\) and \(D_2\).

\begin{definition}[Boundary divisors on \(\mathrm{GPB}\)]
A point of \(\operatorname{GPB}_{\tilde X_0}\) is a pair \((E,S)\) with \(E\) a rank-2, degree-1 bundle on \(\tilde X_0\) and \(S\subset E_{x^+}\oplus E_{x^-}\) a 2-dimensional subspace. Generically \(S\) is the graph of an isomorphism \(A\colon E_{x^+}\xrightarrow{\ \sim\ } E_{x^-}\), and the two boundary divisors record the two ways this graph description can degenerate:
\begin{itemize}
    \item $D_1$ is the closure in $\operatorname{GPB}_{\tilde X_0}$ of the locus of pairs $(E,A)$ with $A\colon E_{x^+}\to E_{x^-}$ of rank at most $1$;
    \item $D_2$ is the closure of the locus of pairs $(E,A')$ with $A'\colon E_{x^-}\to E_{x^+}$ of rank at most $1$.
\end{itemize}
(Taking closures matters only at points where the relevant map fails to be defined at all  i.e.\ where $S$ fails to project isomorphically onto $E_{x^+}$, resp.\ $E_{x^-}$; within the locus where $A$ or $A'$ is defined, rank $\le 1$ is already a closed condition.) We describe $D_1$ in detail; $D_2$ is identical with $x^+$ and $x^-$ exchanged.

On the open locus where \(A\) has rank exactly \(1\), \(\ker A\subset E_{x^+}\) and \(\operatorname{im}A\subset E_{x^-}\) are lines, so \((E,A)\) determines a \emph{quasi-parabolic} bundle on \(\tilde X_0\): a bundle together with a line in each of \(E_{x^+}\) and \(E_{x^-}\). Let \(\operatorname{PBun}_{2,1}(\tilde{X}_0; x^+, x^-)\) denote the moduli stack of such quasi-parabolic bundles, with tautological line subbundles \(\mathbb{L}_+ \subset \mathbb{E}_{x^+}\) and \(\mathbb{L}_- \subset \mathbb{E}_{x^-}\). For parabolic weights \(\epsilon\) small and equal at both points, semistability forces \(\ker A\) and \(\operatorname{im}A\) to be exactly the parabolic lines, so \(A\) induces an isomorphism \(E_{x^+}/\ker A\xrightarrow{\ \sim\ }\operatorname{im}A\); a point of \(D_1\) away from \(D_1\cap D_2\) is thus the same data as a quasi-parabolic bundle together with such an isomorphism:
\[
D_1 \setminus (D_1 \cap D_2) \ \cong \ \Hom(\mathbb{E}_{x^+}/\mathbb{L}_+, \mathbb{L}_-),
\]
and symmetrically
\[
D_2 \setminus (D_1 \cap D_2) \ \cong \ \Hom(\mathbb{E}_{x^-}/\mathbb{L}_-, \mathbb{L}_+).
\]

Fiberwise over $\operatorname{Bun}_{\tilde X_0}$, $D_1$ is the affine cone of rank-$\le1$ maps $E_{x^+}\to E_{x^-}$, and $D_1\cap D_2$ is exactly its boundary at infinity: as a fixed rank-$1$ map $A_0$ is rescaled to infinity, the corresponding plane $S$ degenerates to the \emph{decomposable} plane $\ker A_0\oplus\operatorname{im}A_0$  a line in $E_{x^+}$ direct-summed with a line in $E_{x^-}$, with no isomorphism or homomorphism between them at all. So $D_1\cap D_2$ is parametrized by an independent choice of line in $E_{x^+}$ and in $E_{x^-}$, i.e.\ a full flag at \emph{both} marked points:
\[
D_1 \cap D_2 \ \simeq \ \operatorname{PBun}_{2,1}(\tilde{X}_0;\, x^+,\, x^-).
\]
\end{definition}

The vertex of this same cone  where $A$ drops all the way to rank $0$  is a different locus entirely, disjoint from $D_1\cap D_2$; it is what we blow up next to build the normalization of $\mathcal M_{X_0}$.

To construct the normalization of the Gieseker moduli space, we blow up the variety \(\mathrm{GPB}\) along the two codimension-4 subvarieties:
\begin{itemize}
    \item Let \(W_1 := \{(E, A) \in D_1 \mid \operatorname{rank} A = 0\}\), the vertex of the cone described above, where $S=E_{x^+}\oplus 0$. This is a codimension-4 subvariety of \(\mathrm{GPB}\) and the singular locus of $D_1$. Since $A=0$ pins down $S$ uniquely with no further choice of line, the forgetful morphism \(W_1 \to \Bun_{\tilde{X}_0}\), \((E,A) \mapsto E\), is an isomorphism.
    \item Let \(W_2 := \{(E, A') \in D_2 \mid \operatorname{rank} A' = 0\}\), the analogous vertex inside $D_2$, where $S=0\oplus E_{x^-}$. Likewise $W_2$ is a codimension-4 subvariety and the singular locus of $D_2$, and the forgetful morphism \(W_2 \to \Bun_{\tilde{X}_0}\) is an isomorphism.
\end{itemize}

\begin{construction}[Normalization via successive blow-ups]
We resolve the two cone singularities identified above by blowing up their vertices. First, blow up \(\mathrm{GPB}\) along $W_1$, the rank-0 locus in the $D_1$-direction. Since $D_1$ is fiberwise a cone over $\mathbb P^1\times\mathbb P^1$ with vertex $W_1$, this blow-up replaces $W_1$ with that $\mathbb P^1\times\mathbb P^1$ of $(\ker,\operatorname{im})$-pairs, and the exceptional divisor is
\[
E_1 = \mathbb{P}\Hom(\mathbb{E}_{x^+}, \mathbb{E}_{x^-}).
\]
Symmetrically, blow up along \(W_2\) (the rank-0 locus in the $D_2$-direction) to obtain the exceptional divisor
\[
E_2 = \mathbb{P}\Hom(\mathbb{E}_{x^-}, \mathbb{E}_{x^+}).
\]
Write \(\widetilde{\operatorname{GPB}}_{\tilde{X}_0}\) for the space obtained after both blow-ups, and $\overline{D_1}, \overline{D_2}$ for the strict transforms of the two boundary divisors $D_1$ and $D_2$.

Each blow-up only touches the divisor whose cone it resolves  $E_1$ meets $\overline{D_1}$ but not $\overline{D_2}$, and $E_2$ meets $\overline{D_2}$ but not $\overline{D_1}$. Concretely, the loci
\[
\overline{D_1}\cap E_1 \ \cong \ \overline{D_2}\cap E_2 \ \cong \ \operatorname{PBun}_{2,1}(\tilde{X}_0; x^+, x^-)
\]
arise as the exceptional $\mathbb P^1\times\mathbb P^1$ of the respective blown-up cone point. Together with the boundary-at-infinity locus $\overline{D_1}\cap \overline{D_2}\cong\operatorname{PBun}_{2,1}(\widetilde X_0;x^+,x^-)$ identified earlier, this means $\overline{D_1}$ is now a $\mathbb P^1$-bundle over $\operatorname{PBun}_{2,1}(\tilde X_0;x^+,x^-)$ with two natural sections, $\overline{D_1}\cap \overline{D_2}$ and $\overline{D_1}\cap E_1$  and symmetrically for $\overline{D_2}$.
\end{construction}

\begin{theorem}
The blown-up space \(\widetilde{\mathcal M_{\operatorname{GPB}}}\) is the normalization of the moduli space \(\mathcal{M}_{X_0} = \mathcal{G}_{X_0}(2,1)\) of stable rank-2, degree-1 Gieseker vector bundles on the nodal curve \(X_0\).
\end{theorem}

On \(\widetilde{\mathcal M_{\operatorname{GPB}}}\) we now have four boundary divisors, meeting in a connected chain $E_1 \cup \overline{D_1} \cup \overline{D_2} \cup E_2$ via the identifications above. We write
\[
D = \overline{D_1} \cup \overline{D_2} \cup E_1 \cup E_2
\]
for their union, the total boundary divisor on \(\widetilde{\mathcal M_{\operatorname{GPB}}}\).

\smallskip

\textbf{For convenience, we describe the natural involution on the normalization of the boundary
divisor at the level of stacks; the same description holds verbatim for the corresponding
moduli schemes. The existence of a genuine involution with these properties is established in
\cite[\S5.1--5.3]{Kollar2013}, so here we only record its effect stratum by stratum.}

\smallskip

\subsection{Description of the involution \texorpdfstring{$\iota\colon \widetilde{D}\to \widetilde{D}$}{iota: D-tilde to D-tilde}}

Having described all the boundary divisors, we now turn to the natural involution on their normalization
\[
\widetilde{D} \;=\; D_1 \sqcup D_2 \sqcup E_1 \sqcup E_2,
\]
where, as before, $D_1$ and $D_2$ denote the strict transforms of the two rank-one loci and $E_1, E_2$ the exceptional divisors over the blown-up vertices. There is a natural involution
\[
\iota \colon \widetilde{D} \longrightarrow \widetilde{D}
\]
(see \cite[Section~5.1--5.3]{Kollar2013}), realized canonically by Hecke modifications (elementary transformations) at the marked points $x^+$ and $x^-$; each modification preserves rank, degree, and Gieseker stability, so $\iota$ is well defined on all of $\widetilde D$. Since $\widetilde D$ is a disjoint union of four pieces, $\iota$ must be pinned down stratum by stratum. Two features of the resulting description are worth flagging in advance: first, $D_1\cap D_2$ occurs as \emph{two} distinct loci in $\widetilde D$  one inside $D_1$, one inside $D_2$  and these two copies map to different targets ($E_2\cap D_2$ and $E_1\cap D_1$, respectively); second, on the open dense strata of $D_1,D_2$ and of $E_1,E_2$, the involution swaps $D_1\leftrightarrow D_2$ and $E_1 \leftrightarrow E_2$ as expected, matching the two exceptional/rank-one divisors that arose from the same blow-up construction.

\begin{itemize}

\item \textbf{The generic strata: $\iota\colon D_1\setminus\big[(D_1\cap D_2)\cup(D_1\cap E_1)\big]\ \longrightarrow\ D_2\setminus\big[(D_1\cap D_2)\cup(D_2\cap E_2)\big]$.}

A point of $D_1\setminus(D_1\cap D_2)$ is a pair $(E,A)$ with $A\colon \mathbb{E}_{x^+}\to \mathbb{E}_{x^-}$ a rank-one homomorphism. Set
\[
L_+ := \ker A \subset \mathbb{E}_{x^+}, \qquad L_- := \operatorname{im} A \subset \mathbb{E}_{x^-},
\]
so that $A$ induces an isomorphism $\mathbb{E}_{x^+}/L_+ \xrightarrow{\ \sim\ } L_-$. The involution is realized by a pair of complementary Hecke modifications at $x^+$ and $x^-$ that jointly preserve the total degree:
\[
\begin{tikzcd}
E_{x^+} \arrow[d, "A"'] & E''_{x^+} \arrow[l, "H_1"'] \\
E_{x^-} \arrow[r, "H_2"'] & E''_{x^-} \arrow[u, "A''"']
\end{tikzcd}
\]
subject to
\[
\operatorname{im} H_1 = \ker A, \qquad \operatorname{im} A = \ker H_2, \qquad \ker H_1 = \operatorname{im} A'', \qquad \ker A'' = \operatorname{im} H_2 .
\]
The image of $(E,A)$ under $\iota$ is defined to be $\iota(E,A) := (E'', A'')$. In terms of flags, this sends
\[
(E,\, L_+,\, L_-) \ \longmapsto\ \big(E'',\ \mathbb{E}_{x^+}/L_+,\ \mathbb{E}_{x^-}/L_-\big),
\]
with $L_+ = \ker A$ and $L_- = \operatorname{im} A$ as above.

\item \textbf{From $D_1\cap D_2$ (viewed inside $D_1$) to $E_2\cap D_2$.}

A point of $D_1\cap D_2 \subset D_1$ is a triple $(E,\, L_+\subset E_{x^+},\, L_-\subset E_{x^-})$. Perform the Hecke modification that decreases degree at $x^+$,
\[
0 \to E' \to E \to E_{x^+}/L_+ \to 0,
\]
followed by the one that increases degree at $x^-$,
\[
0 \to E' \to E'' \to \mathbb{C} \to 0,
\]
where the induced map $E'_{x^-}\to E''_{x^-}$ has kernel $L_-$ and image isomorphic to $E'_{x^-}/L_-$. The new flags are $E_{x^+}/L_+ \subset E''_{x^+}$ and $E_{x^-}/L_- \subset E''_{x^-}$, so
\[
\iota\big(E,\, L_+,\, L_-\big) = \big(E'',\ E_{x^+}/L_+,\ E_{x^-}/L_-\big) \in E_2\cap D_2 .
\]

\item \textbf{From $D_1\cap D_2$ (viewed inside $D_2$) to $E_1\cap D_1$.}

Symmetrically, a point of $D_1\cap D_2 \subset D_2$ is again a triple $(E,\, L_+\subset E_{x^+},\, L_-\subset E_{x^-})$, and this time we decrease degree at $x^-$,
\[
0 \to E' \to E \to E_{x^-}/L_- \to 0,
\]
then increase degree at $x^+$,
\[
0 \to E' \to E'' \to \mathbb{C} \to 0,
\]
where the induced map $E'_{x^+}\to E''_{x^+}$ has kernel $L_+$ and image isomorphic to $E'_{x^+}/L_+$. As before the new flags are $E_{x^+}/L_+ \subset E''_{x^+}$ and $E_{x^-}/L_- \subset E''_{x^-}$, giving
\[
\iota\big(E,\, L_+,\, L_-\big) = \big(E'',\ E_{x^+}/L_+,\ E_{x^-}/L_-\big) \in E_1\cap D_1 .
\]

\item \textbf{The generic strata of the exceptional divisors: $\iota\colon E_1\setminus (D_1\cap E_1)\ \longrightarrow\ E_2\setminus (D_2\cap E_2)$.}

A point of $E_1 = \mathbb{P}\operatorname{Hom}(\mathbb{E}_{x^+}, \mathbb{E}_{x^-})$ is a line $\ell = \langle \phi \rangle$, where $\phi\colon \mathbb{E}_{x^+}\to \mathbb{E}_{x^-}$ is an isomorphism. Let $L := \mathcal{O}_{\tilde X_0}(x^- - x^+)$, a degree-zero line bundle on $\tilde X_0$. The involution is induced by twisting by $L$:
\[
(E,\phi)\ \longmapsto\ (E\otimes L,\ \phi'), \qquad
\phi' \colon \mathbb{E}_{x^-}\otimes L_{x^-} \longrightarrow \mathbb{E}_{x^+}\otimes L_{x^+},
\]
where $\phi'$ is the inverse of $\phi$ after twisting by $L$. This produces a well-defined line in $\mathbb{P}\operatorname{Hom}(\mathbb{E}_{x^-}, \mathbb{E}_{x^+})$, i.e.\ a point of $E_2$, and the construction gives a canonical algebraic isomorphism
\[
E_1\setminus (D_1\cap E_1) \ \simeq\ E_2\setminus (D_2\cap E_2).
\]

\item \textbf{From $E_1\cap D_1$ to $E_2\cap D_2$.}

Both loci are canonically isomorphic to the moduli space $\operatorname{PBun}_{2,1}(\tilde X_0; x^+, x^-)$ of rank-2, degree-1 parabolic bundles with full flags $(L_+\subset \mathbb{E}_{x^+},\, L_-\subset \mathbb{E}_{x^-})$. On such a parabolic bundle $(E, L_+, L_-)$, the involution is realized by the same twist as above, applied to both flags:
\begin{itemize}
    \item first, $E \mapsto E\otimes \mathcal{O}(x^- - x^+)$;
    \item at $x^+$: fix any identification $\lambda_+\colon \mathbb{C}\to \mathcal{O}(x^- - x^+)_{x^+}$, and let the new flag be the image of $\mathrm{id}_{E_{x^+}}\otimes \lambda_+ \colon E_{x^+} \to E_{x^+}\otimes \mathcal{O}(x^- - x^+)_{x^+}$;
    \item at $x^-$: fix any identification $\lambda_-\colon \mathbb{C}\to \mathcal{O}(x^- - x^+)_{x^-}$, and let the new flag be the image of $\mathrm{id}_{E_{x^-}}\otimes \lambda_- \colon E_{x^-} \to E_{x^-}\otimes \mathcal{O}(x^- - x^+)_{x^-}$.
\end{itemize}
This gives the canonical isomorphism $E_1\cap D_1 \simeq E_2\cap D_2$, and preserves parabolic stability for sufficiently small weights.

\end{itemize}

Altogether, $\iota$ defines an algebraic automorphism of order two on the total divisor $\widetilde{D}$.

\smallskip

\textbf{In the subsection that follows, we work with the moduli \emph{schemes} rather than the
moduli stacks; however, for notational simplicity we retain the stack notation throughout,
writing, e.g., $\operatorname{Bun}_{\tilde X_0}$ for $\mathcal M_{\operatorname{Bun}}$, and
similarly for $\operatorname{GPB}_{\tilde X_0}$, $\widetilde{\operatorname{GPB}}_{\tilde X_0}$,
and $\mathrm{GVB}$.}

\smallskip

\subsection{Description of the effect of the involution on various line bundles}

\subsubsection{\textbf{Description of \(\mathcal{O}_Y(D_1)\)}}

From the blow-up construction,
\[
\mathcal{O}_Y(D_1)
\;\cong\; \pi^*\mathcal{O}_{\mathrm{GPB}}(D_1)\otimes \mathcal{O}_Y(-2E_1)
\;\cong\; \pi^*\mathcal{O}_{\mathrm{GPB}/\Bun}(1)\otimes \det(\mathbb{E}_{x^+})^{-1}\otimes \mathcal{O}_Y(-2E_1).
\]
Indeed, the point of $\mathrm{GPB}$ over which $E_1$ is the exceptional divisor lies on $D_1$ with multiplicity $2$ (local equation $xw-yz=0$), so
\[
\pi^*\mathcal{O}_{\mathrm{GPB}}(D_1) \;\cong\; \mathcal{O}_Y(D_1)\otimes \mathcal{O}_Y(2E_1).
\]
To identify $\mathcal{O}_{\mathrm{GPB}}(D_1)$ itself, recall the universal short exact sequence on $\mathrm{GPB}$,
\[
0 \to S \to \mathbb{E}_{x^+}\oplus \mathbb{E}_{x^-} \to Q \to 0,
\]
where $S$ is the tautological rank-2 subbundle. The divisor $D_1$ is precisely the vanishing locus of the determinant of the induced map $S \to \mathbb{E}_{x^-}$, so
\[
\mathcal{O}_{\mathrm{GPB}}(D_1)
= (\det S)^{-1}\otimes \det \mathbb{E}_{x^-}
= \det Q \otimes (\det \mathbb{E}_{x^-}\otimes\det \mathbb{E}_{x^+})^{-1}\otimes \det \mathbb{E}_{x^-}
= \mathcal{O}_{\mathrm{GPB}/\Bun}(1)\otimes (\det \mathbb{E}_{x^+})^{-1},
\]
using $\det S \otimes \det Q \cong \det(\mathbb{E}_{x^+}\oplus \mathbb{E}_{x^-})$ and the convention $\mathcal{O}_{\mathrm{GPB}/\Bun}(1) = \det Q$.

\subsubsection{\textbf{Description of \(\mathcal{O}_Y(D_2)\)}}

By the symmetric argument,
\[
\mathcal{O}_Y(D_2)
\;\cong\; \pi^*\mathcal{O}_{\mathrm{GPB}}(D_2)\otimes \mathcal{O}_Y(-2E_2)
\;\cong\; \pi^*\mathcal{O}_{\mathrm{GPB}/\Bun}(1)\otimes \det(\mathbb{E}_{x^-})^{-1}\otimes \mathcal{O}_Y(-2E_2),
\]
since the point of $\operatorname{GPB}_{\tilde X_0}$ blown up to produce $E_2$ lies on $D_2$ with multiplicity $2$, giving $\pi^*\mathcal{O}_{\mathrm{GPB}}(D_2) \cong \mathcal{O}_Y(D_2)\otimes \mathcal{O}_Y(2E_2)$. Here $D_2$ is the vanishing locus of the determinant of the induced map $S \to \mathbb{E}_{x^+}$, so
\[
\mathcal{O}_{\mathrm{GPB}}(D_2)
= (\det S)^{-1}\otimes \det \mathbb{E}_{x^+}
= \det Q \otimes (\det \mathbb{E}_{x^-}\otimes\det \mathbb{E}_{x^+})^{-1}\otimes \det \mathbb{E}_{x^+}
= \mathcal{O}_{\mathrm{GPB}/\Bun}(1)\otimes (\det \mathbb{E}_{x^-})^{-1}.
\]

\subsubsection{\textbf{Description of \(\pi^*\mathcal{O}_{\mathrm{GPB}/\Bun}(1)|_{D_1}\)}}

We first describe $D_1$ itself. On the parabolic bundle moduli space $\mathrm{PBun}\to \Bun$, write $L_+\subset \mathbb{E}_{x^+}$ and $L_-\subset \mathbb{E}_{x^-}$ for the two tautological line subbundles. Then
\[
D_1 \;\cong\; \mathbb{P}\bigl(\mathcal{O} \oplus \operatorname{Hom}(\mathbb{E}_{x^+}/L_+,\, L_-)\bigr),
\]
so that
\[
\operatorname{Pic}(D_1)
\;\cong\; \operatorname{Pic}(\Bun) \oplus \mathbb{Z}\cdot [\mathbb{E}_{x^+}/L_+] \oplus \mathbb{Z}\cdot [\mathbb{E}_{x^-}/L_-] \oplus \mathbb{Z}\cdot \mathcal{O}_{D_1/\mathrm{PBun}}(1).
\]

\subsubsection{\textbf{Description of \(\pi^*\mathcal{O}_{\mathrm{GPB}/\Bun}(1)|_{D_2}\)}}

Analogously,
\[
D_2 \;\cong\; \mathbb{P}\bigl(\mathcal{O} \oplus \operatorname{Hom}(\mathbb{E}_{x^-}/L_-,\, L_+)\bigr),
\]
and
\[
\operatorname{Pic}(D_2)
\;\cong\; \operatorname{Pic}(\Bun) \oplus \mathbb{Z}\cdot [\mathbb{E}_{x^-}/L_-] \oplus \mathbb{Z}\cdot [\mathbb{E}_{x^+}/L_+] \oplus \mathbb{Z}\cdot \mathcal{O}_{D_2/\mathrm{PBun}}(1).
\]

\subsubsection{\textbf{Description of \(\pi^*\mathcal{O}_{\mathrm{GPB}/\Bun}(1)|_{E_1}\) and \(|_{E_2}\)}}

Recall $E_1 \cong \mathbb{P}(\operatorname{Hom}(\mathbb{E}_{x^+}, \mathbb{E}_{x^-}))$ is a $\mathbb{P}^3$-bundle over $\Bun$, so
\[
\operatorname{Pic}(E_1) \;\cong\; \operatorname{Pic}(\Bun) \oplus \mathbb{Z}\cdot \mathcal{O}_{E_1/\Bun}(1).
\]
Similarly, $E_2 \cong \mathbb{P}(\operatorname{Hom}(\mathbb{E}_{x^-}, \mathbb{E}_{x^+}))$ and
\[
\operatorname{Pic}(E_2) \;\cong\; \operatorname{Pic}(\Bun) \oplus \mathbb{Z}\cdot \mathcal{O}_{E_2/\Bun}(1).
\]

\subsubsection{\textbf{Restrictions of the exceptional divisors}}

Recall from above that
\[
\operatorname{Pic}(D_1)
\;\cong\; \operatorname{Pic}(\Bun) \oplus \mathbb{Z}\cdot [\mathbb{E}_{x^+}/L_+] \oplus \mathbb{Z}\cdot [\mathbb{E}_{x^-}/L_-] \oplus \mathbb{Z}\cdot \mathcal{O}_{D_1/\mathrm{PBun}}(1).
\]
Since $D_1 \cap E_1$ is the section at infinity and $D_1 \cap D_2$ is the zero section of the $\mathbb{P}^1$-bundle $D_1 \to \mathrm{PBun}$, restricting the exceptional divisor gives
\[
\mathcal{O}_Y(E_1)|_{D_1} \;\cong\; \mathcal{O}_{D_1}(D_1 \cap E_1) \;\cong\; \mathcal{O}_{D_1/\mathrm{PBun}}(1),
\]
and symmetrically,
\[
\mathcal{O}_Y(E_2)|_{D_2} \;\cong\; \mathcal{O}_{D_2}(D_2 \cap E_2) \;\cong\; \mathcal{O}_{D_2/\mathrm{PBun}}(1).
\]
Since $D_2 \cap E_1 = \emptyset$ and $D_1 \cap E_2 = \emptyset$, the cross-restrictions are trivial:
\[
\mathcal{O}_Y(E_1)|_{D_2} \;\cong\; \mathcal{O}_{D_2}, \qquad \mathcal{O}_Y(E_2)|_{D_1} \;\cong\; \mathcal{O}_{D_1}.
\]

\subsubsection{\textbf{Description of \(\iota^*(\mathcal{O}_Y(E_1)|_{D_1})\)}}

Since $\iota$ carries $D_1\cap D_2$ (viewed inside $D_2$) to $E_1\cap D_1$, pulling back along $\iota$ gives
\begin{align*}
\iota^*\bigl(\mathcal{O}_Y(E_1)|_{D_1}\bigr)
&= \iota^*\bigl(\mathcal{O}_{D_1}(D_1 \cap E_1)\bigr) \\
&= \mathcal{O}_{D_2}(D_1 \cap D_2) \\
&= \mathcal{O}_{D_2}(E_2 \cap D_2)\otimes \pi^*\operatorname{Hom}(L_+,\, \mathbb{E}_{x^-}/L_-) \\
&= \mathcal{O}_Y(E_2)|_{D_2}\otimes \pi^*\operatorname{Hom}(L_+,\, \mathbb{E}_{x^-}/L_-).
\end{align*}

\subsubsection{\textbf{Description of \(\iota^*(\mathcal{O}_Y(E_2)|_{D_2})\)}}

Symmetrically, since $\iota$ carries $D_1\cap D_2$ (viewed inside $D_1$) to $E_2\cap D_2$,
\begin{align*}
\iota^*\bigl(\mathcal{O}_Y(E_2)|_{D_2}\bigr)
&= \iota^*\bigl(\mathcal{O}_{D_2}(D_2 \cap E_2)\bigr) \\
&= \mathcal{O}_{D_1}(D_1 \cap D_2) \\
&= \mathcal{O}_{D_1}(E_1 \cap D_1)\otimes \pi^*\operatorname{Hom}(L_-,\, \mathbb{E}_{x^+}/L_+) \\
&= \mathcal{O}_Y(E_1)|_{D_1}\otimes \pi^*\operatorname{Hom}(L_-,\, \mathbb{E}_{x^+}/L_+).
\end{align*}

Combining the two computations, the involution acts on the restrictions of the total exceptional divisor $E_1+E_2$ by
\begin{align*}
\iota^*\Bigl(\bigl[\mathcal{O}_Y(E_1+E_2)|_{D_1},\ \mathcal{O}_Y(E_1+E_2)|_{D_2}\bigr]\Bigr)
&= \iota^*\Bigl(\bigl[\mathcal{O}_Y(E_1)|_{D_1},\ \mathcal{O}_Y(E_2)|_{D_2}\bigr]\Bigr) \\
&= \Bigl[\ \mathcal{O}_Y(E_1)|_{D_1}\otimes \pi^*\operatorname{Hom}(L_-,\, \mathbb{E}_{x^+}/L_+),\\
&\qquad\ \mathcal{O}_Y(E_2)|_{D_2}\otimes \pi^*\operatorname{Hom}(L_+,\, \mathbb{E}_{x^-}/L_-)\ \Bigr].
\end{align*}

\subsection{Description of the Theta Bundle on the moduli stack}
 
We first construct the theta bundle at the level of stacks; this construction is the input for the moduli-space-level description given in the next subsection. Throughout, consider the diagram of stacks
\[
\begin{tikzcd}
\mathrm{Gi}:=\mathrm{GVB} \arrow[d, "\pi"] & \widetilde{G}:=\widetilde{\mathrm{GPB}} \arrow[l, "\nu"]\arrow[d, "p_G"] \arrow[dd, bend left=60, "p_{B}"] \arrow[ddd, bend left=90, "p_{P}"] \\
\text{TF} & {G}:={\mathrm{GPB}} \arrow[l, "\nu_{G}"]\arrow[d, "p_{G/B}"]\\
& B:=\mathrm{Bun}\arrow[d, "p_{B/P}"]\\
& P:=\mathrm{Pic}^1_{\tilde{X}_0}
\end{tikzcd}
\]
in which the Gieseker moduli stack $\mathrm{GVB}$ is realized as the blow-up of $G = \mathrm{GPB}$ along two loci, and $\widetilde{G}$ is the corresponding blow-up of $G$. The maps in the diagram are:
\begin{itemize}
    \item \( p_G \colon \widetilde{G} \to G \), the projection of the Grassmannian bundle;
    \item \( p_B \colon \widetilde{G} \to B \) (and, by composition, \( p_B \colon G \to B \)), the map to $B$;
    \item \( \nu_G \colon G \to \mathrm{TF} \), the natural map to the moduli space of torsion-free sheaves;
    \item \( \nu \colon \widetilde{G} \to GVB \), the normalization map.
\end{itemize}
 
On $G = \mathrm{GPB}$ we have the universal short exact sequence
\[
0 \to S \to p_B^* (E_{x^+}\oplus E_{x^-}) \to Q \to 0,
\]
where $S$ is the universal rank-$2$ subbundle and $Q$ is the universal quotient. Comparing theta bundles across $\nu_G$, the theta bundle on the moduli space of torsion-free sheaves satisfies
\[
\nu_G^* \theta_{\mathrm{TF}} \ \cong \ p_{G/B}^* \theta_B \otimes \det Q \ = \ p_{G/B}^* \theta_B \otimes \mathcal{O}_{G/B}(1).
\]
Guided by this formula, we \emph{define} the theta bundle on $\widetilde{G}$ by pulling the same expression back to $\widetilde{G}$:
\[
p_B^*\theta_B\otimes p_G^*\mathcal{O}_{G/B}(1).
\]
 
It remains to relate $p_B^*\theta_B\otimes p_G^*\mathcal{O}_{G/B}(1)$ to the normalization of the Gieseker moduli space itself. Let $\mathcal E$ denote the universal Gieseker vector bundle on the universal modification family $f :\mathfrak X\to GVB$, and define the theta bundle on $GVB$ by $\theta_{Gie}:=\det Rf_*\mathcal E$. A Gieseker vector bundle does not have higher direct images on the modification map $\mathfrak X\to X_0\times GVB$, so $\theta_{Gie}\cong \pi^*\theta_{TF}$, and hence $\nu^*\theta_{Gie}\cong p_B^*\theta_B\otimes p_G^*\mathcal{O}_{G/B}(1)$. This completes the description of the theta bundle on the normalization of the Gieseker moduli stack.
 
\subsection{Description of the Theta Bundle on the moduli space}
 
The construction above pins down $\theta_B$ only abstractly, as a bundle of determinants of cohomology twisted by an auxiliary bundle chosen to kill its weight. We now make $\theta_B$ fully explicit, using the following Drezet--Narasimhan pullback formula for the ample generator of the Picard group of a fixed-determinant moduli space.
 
\smallskip

\begin{proposition}\label{prop:DN-theta-pullback}\cite[Proposition 2.1.]{Narasimhan2017}
Let $F$ be a family of semistable vector bundles on $X$ of rank $r$ and fixed determinant $\xi$ of degree $d$, parametrized by a smooth variety $T$. Let $M(r,\xi)$ be the moduli space of semistable bundles of rank $r$ and determinant $\xi$ on $X$, and let $\varphi\colon T\to M(r,\xi)$ be the morphism determined by $F$. Let further $p$ be a point of $X$, and let $F_p$ be the bundle on $T$ obtained by restricting $F$ to $p\times T$. Then the pullback by $\varphi$ of the ample generator $\theta$ of $Pic(M(r,\xi))$ is isomorphic to
\[
\varphi^*\theta\ \cong\ (Det F)^{r/n}\otimes(\det F_p)^{\chi/n},
\]
where $n := \gcd(r,d)$ and $\chi := d+r(1-g)$.
\end{proposition}

In our setting $r=2$, $d=1$, and the genus of $\tilde X_0$ is $g-1$, so $n=\gcd(2,1)=1$ and
\[
\chi = 1+2\bigl(1-(g-1)\bigr) = 1+2(2-g).
\]
Therefore, Proposition~\ref{prop:DN-theta-pullback} gives
\[
\varphi^*\theta_{\mathcal M_{\mathrm{VB}}}\ \cong\ (Det E)^{2}\otimes(det E_p)^{1+2(2-g)}.
\]

\

\
 
We now return to the diagram relating $\mathrm{Gi}$, $\widetilde G$, $G$, $B$, and $P$ but now we consider moduli spaces instead of stacks:
\[
\begin{tikzcd}
\mathcal M_{\mathrm{GVB}} \arrow[d, "\pi"] & \widetilde{G}:=\widetilde{\mathcal M_{\operatorname{GPB}}} \arrow[l, "\nu"]\arrow[d, "p_G"] \arrow[dotted,dd, bend left=60, "p_{B}"] \arrow[dotted,ddd, bend left=90, "p_{P}"] \\
\mathcal M_{\mathrm{TF}} & G:=\mathcal M_{\operatorname{GPB}} \arrow[l, "\nu_{G}"]\arrow[dotted, d, "p_{G/B}"]\\
& B:=\mathcal M_{\operatorname{Bun}}\arrow[d, "p_{B/P}"]\\
& P:=\mathcal Pic^1_{\tilde{X}_0}
\end{tikzcd}
\]

First of all, the dotted maps are only rational maps, i.e., they are only defined in the complement of certain closed subsets which are always of codimension greater than equal to $2$. Applying Proposition~\ref{prop:DN-theta-pullback} on $\mathcal M_{\mathrm{TF}}$ and on $B$ in turn gives $\theta_{\mathcal M_{\mathrm{TF}}}\cong (Det ~\mathcal F)^{2}\otimes(\det \mathcal F_p)^{1+2(1-g)}$ and $\theta_{\mathcal M_{\mathrm{VB}}}\ \cong\ (Det ~E)^{2}\otimes(\det E_p)^{1+2(2-g)}$. Since $\nu^*Det~\mathcal F\cong Det ~E\otimes \det Q$, and $\pi^*\theta_{\mathcal M_{\mathrm{TF}}}\cong \theta_{\mathcal M_{\mathrm{GVB}}}$, pulling the first formula back along $\nu$ yields the following result.

\begin{equation}
\nu^*\theta_{\mathcal M_{\mathrm{GVB}}}\cong p_P^*(Det ~E)^2\otimes p_G^*(\det Q)^2\otimes p_P^*(\det E_p)^{1+2(1-g)}=p_P^*\theta_{\mathcal M_{\mathrm{VB}}}\otimes p_G^*(\det Q)^2\otimes p_P^*(\det E_p)^{-2}.
\end{equation}

\

We set $\theta_{\widetilde{G}}:=p_P^*\theta_{\mathcal M_{\mathrm{VB}}}\otimes p_G^*(\det Q)^2\otimes p_P^*(\det E_p)^{-2}$.

\smallskip

\textbf{In the subsection that follows, we work with the moduli \emph{schemes} rather than the
moduli stacks.}

\subsection{Computation of the Picard Group}
 
The Picard group of $\widetilde{\mathcal M_{{\operatorname{GPB}}}}$ is generated by classes pulled back from $\mathcal Pic^1(\tilde{X}_0)$, together with the theta line bundle $\theta_{\tilde G}$ and the boundary and exceptional classes
\[
\delta := p_G^* \mathcal{O}_{G/B}(1),\qquad
\varepsilon_1 := \mathcal{O}_{\widetilde{G}}(E_1),\qquad
\varepsilon_2 := \mathcal{O}_{\widetilde{G}}(E_2).
\]
A general class therefore has the form
\[
\mathcal{L} \ \equiv \ L + a_1\theta_{\tilde G} + b_1\delta + c_1\varepsilon_1 + c_2\varepsilon_2,
\qquad L\in Pic^1(\tilde X_0),\ \ a_1,b_1, c_1, c_2 \in \mathbb{Z}.
\]
 
The Picard group of $\mathcal{M}_{X_0}$ sits within that of $\widetilde{\mathcal M_{{\operatorname{GPB}}}}$ via the short exact sequence
\[
0 \to \operatorname{Pic}(\mathcal{M}_{X_0}) \longrightarrow \operatorname{Pic}(\widetilde{\mathcal M_{{\operatorname{GPB}}}})
\xrightarrow{\ \operatorname{res}_{\widetilde{D}} - \iota^* \circ \operatorname{res}_{\widetilde{D}}\ }
\operatorname{Pic}(\widetilde{D}),
\]
where $\widetilde{D} = D_1 \sqcup D_2 \sqcup E_1 \sqcup E_2$ is the normalization of the total boundary divisor, and the displayed map sends a line bundle $\mathcal{L}$ on $\widetilde{\mathcal M_{{\operatorname{GPB}}}}$ to the difference
\[
\mathcal{L}|_{\widetilde{D}} - \iota^*(\mathcal{L}|_{\widetilde{D}})\ \in \ \operatorname{Pic}(D_1) \oplus \operatorname{Pic}(D_2) \oplus \operatorname{Pic}(E_1) \oplus \operatorname{Pic}(E_2).
\]
The kernel of this map consists precisely of the $\iota$-invariant line bundles on $\widetilde{\mathcal M_{{\operatorname{GPB}}}}$  exactly those classes that descend to $\mathcal{M}_{X_0}$. So $\operatorname{Pic}(\mathcal{M}_{X_0})$ is identified with the subgroup of $\iota$-invariant classes in $\operatorname{Pic}(\widetilde{\mathcal M_{{\operatorname{GPB}}}})$.
 
Since $\theta_{\widetilde{G}}$ already descends to $\mathcal{M}_{X_0}$, it suffices to work out the descent condition for a class of the reduced form
\[
\mathcal L:= \ L + b_1\delta + c_1\varepsilon_1 + c_2\varepsilon_2,
\qquad L \in \operatorname{Pic}^1(\tilde{X}_0),\ \ b_1, c_1, c_2 \in \mathbb{Z}.
\]
Such $\mathcal L$ descends to $\mathcal M_{X_0}$ exactly when it is $\iota$-invariant on the normalization $\widetilde D = D_1 \cup D_2 \cup E_1 \cup E_2$ of the total boundary divisor. We derive this condition in two steps.
 
\begin{itemize}
    \item \textbf{Step 1 (restriction to $D_1$ and $D_2$).} Restricting $\mathcal L$ to $D_1$ and simplifying term by term,
\begin{align*}
\mathcal L\big|_{D_1} &= L\big|_{D_1} + b_1 \delta\big|_{D_1} + c_1 \varepsilon_1\big|_{D_1} + c_2 \varepsilon_2\big|_{D_1}\\
&= L\big|_{D_1} + b_1 \mathcal O_{D_1}(s_0) - b_1 [\det \mathbb{E}_{x^-}] + c_1 \mathcal O_{D_1}(s_{\infty}) \\
&= L\big|_{D_1} + b_1 \mathcal O_{D_1}(s_0) + c_1 \mathcal O_{D_1}(s_{\infty}) - b_1 \ell_2 \\
&= L\big|_{D_1} + b_1 \mathcal O_{D_1}(s_{\infty}) + b_1 \operatorname{Hom}(L_-, \mathbb{E}_{x^+}/L_+) + c_1 \mathcal O_{D_1}(s_{\infty}) - b_1 \ell_2 \\
&= L\big|_{D_1} + (b_1 + c_1) \mathcal O_{D_1}(s_{\infty}) + b_1 \operatorname{Hom}(L_-, \mathbb{E}_{x^+}/L_+) - b_1 \ell_2.
\end{align*}
 
Applying $\iota^*$,
\begin{align*}
\iota^*\bigl(\mathcal L\big|_{D_1}\bigr) &= \iota^*\bigl(L\big|_{D_1}\bigr) + (b_1 + c_1) \mathcal O_{D_2}(s_0') + b_1 \operatorname{Hom}(\mathbb{E}_{x^-}/L_-, L_+) - b_1 \iota^*\ell_2 \\
&= \iota^*\bigl(L\big|_{D_1}\bigr) + (b_1 + c_1) \mathcal O_{D_2}(s'_{\infty}) + (b_1 + c_1) \operatorname{Hom}(L_+, \mathbb{E}_{x^-}/L_-) + b_1 \operatorname{Hom}(\mathbb{E}_{x^-}/L_-, L_+) - b_1 \iota^*\ell_2 \\
&= \iota^*\bigl(L\big|_{D_1}\bigr) + (b_1 + c_1) \mathcal O_{D_2}(s'_{\infty}) + c_1 \operatorname{Hom}(L_+, \mathbb{E}_{x^-}/L_-) - b_1 \iota^*\ell_2.
\end{align*}
 
On the other side, restricting $\mathcal L$ to $D_2$ in the same way gives
\begin{align*}
\mathcal L\big|_{D_2}&=L\big|_{D_2}+b_1\delta\big|_{D_2}+c_2\mathcal O_{D_2}(E_2\cap D_2)\\
&=L\big|_{D_2}+b_1\bigl(p_G^*\mathcal O(D_1)\otimes p_{B}^*\ell_1^{-1}\bigr)\big|_{D_2}+c_2\mathcal O_{D_2}(s'_{\infty})\\
&=L\big|_{D_2}+b_1\mathcal O_{D_2}(D_1\cap D_2)-b_1\ell_1+c_2\mathcal O_{D_2}(s'_{\infty})\\
&=L\big|_{D_2}+b_1\mathcal O_{D_2}(s'_0)-b_1\ell_1+c_2\mathcal O_{D_2}(s'_{\infty})\\
&=L\big|_{D_2}+(b_1+c_2)\mathcal O_{D_2}(s'_{\infty})+b_1\operatorname{Hom}(L_+, \mathbb{E}_{x^-}/L_-)-b_1\ell_1.
\end{align*}
 
Imposing $\iota^*(\mathcal L|_{D_1}) = \mathcal L|_{D_2}$ and comparing term by term along the tower of morphisms $\widetilde{G}\to G\to B\to P$ gives
\[
b_1+c_2=b_1+c_1 \implies c_1=c_2,\ b_1=c_1,
\]
together with
\[
\iota^*\bigl(L\big|_{D_1}\bigr)-L\big|_{D_2}=b_1\bigl(\iota^*\ell_2-\ell_1\bigr).
\]
 
    \item \textbf{Step 2 (restriction to $E_1$ and $E_2$).} Substituting $b_1=c_1=c_2$ from Step 1, restriction of $\mathcal L$ to $E_1$ gives
\begin{align*}
\mathcal L\big|_{E_1} &= L\big|_{E_1} + b_1 \delta\big|_{E_1} + b_1 \mathcal O(E_1)\big|_{E_1} \\
&= L\big|_{E_1} - b_1\ell_2 + b_1 \mathcal O_{E_1}(-H_1) = L\big|_{E_1} - b_1 \ell_2 - b_1 \mathcal O_{E_1/B}(1),
\end{align*}
so that
\begin{align*}
\iota^*\bigl(\mathcal L\big|_{E_1}\bigr) &= \iota^*\bigl(L\big|_{E_1}\bigr) - b_1 \iota^*\mathcal O_{E_1/B}(1) - b_1 \iota^*\ell_2 \\
&= \iota^*\bigl(L\big|_{E_1}\bigr) - b_1 \mathcal O_{E_2/B}(1) - b_1 \iota^*\ell_2,
\end{align*}
while restriction to $E_2$ gives
\[
\mathcal L\big|_{E_2} = L\big|_{E_2} - b_1\ell_1 + b_1 \mathcal O_{E_2}(-H_2) = L\big|_{E_2} - b_1 \mathcal O_{E_2/B}(1) - b_1 \ell_1.
\]
 
Imposing $\iota^*(\mathcal L|_{E_1}) \cong \mathcal L|_{E_2}$ gives
\[
\iota^*\bigl(L\big|_{E_1}\bigr) - L\big|_{E_2}=b_1\bigl(\iota^*\ell_2-\ell_1\bigr).
\]
\end{itemize}
 
Recall that the line bundles $\ell_1,\ell_2$ are pulled back from $\operatorname{Pic}^1(\tilde X_0)$, and the involution $\iota$ acts on $\operatorname{Pic}^1(\tilde{X}_0)$ by the translation
\[
t_{x^--x^+}\colon \operatorname{Pic}^1(\tilde{X}_0)\to \operatorname{Pic}^1(\tilde{X}_0), \qquad L \longmapsto L\otimes \mathcal O_{\tilde{X}_0}(x^--x^+).
\]
Under this translation both $\ell_1$ and $\ell_2$ are invariant, i.e. $\iota^*\ell_i=\ell_i$, so the condition from Step 2 simplifies to
\[
t_{x^--x^+}^*(L) - L=b_1(\ell_2-\ell_1).
\]
 
Combining this with the short exact sequence above, we conclude that $\operatorname{Pic}(\mathcal{M}_{X_0})$ is the subgroup of $\operatorname{Pic}(\widetilde{\mathcal M_{{\operatorname{GPB}}}})$ generated by:
\begin{itemize}
    \item the descended theta line bundle $\theta_{\widetilde{G}}$;
    \item classes $L+b_1(\delta+\varepsilon_1+\varepsilon_2)$ with $L \in \mathcal Pic^1(\tilde X_0)$, $b_1\in\mathbb Z$, satisfying $t_{x^--x^+}^*(L) - L=b_1(\ell_2-\ell_1)$.
\end{itemize}
 
This proves the following theorem.
 
\

\begin{theorem}
The Picard group of $\mathcal{M}_{X_0}$ is the subgroup of $\operatorname{Pic}(\widetilde{\mathcal M_{{\operatorname{GPB}}}})$ generated by the following two types of line bundles:
\[
\theta_{\widetilde{G}}, \text{and}~(L + b(\delta + \varepsilon_1 + \varepsilon_2)),
\]
where $L \in \mathcal Pic^1(\tilde{X}_0)$ and $b \in \mathbb{Z}$ satisfy
\[
t_{x^- - x^+}^*(L) - L = b(\ell_2 - \ell_1).
\]
\end{theorem}

\

\

\section{Virtual Poincare Polynomial}
In this section we compute the virtual Poincaré polynomial of the moduli stack \(GVB\) of Gieseker vector bundles of rank 2 and degree 1 on the nodal curve \(X_0\). We use the Whitney stratification of \(GVB\) and the relation with the moduli stack \(TF\) of torsion-free sheaves via the natural projective birational morphism \(\theta \colon GVB \to TF\). The computation relies on scissor relations, fibration structures, and the Harder--Narasimhan stratification on the torsion-free side.

Virtual Poincaré Polynomial (also called virtual Poincaré series or motivic Poincaré polynomial) is a powerful invariant used in algebraic geometry and topology to study the ``size'' or cohomology of possibly singular or non-compact spaces, especially moduli spaces and stacks.

\subsection*{Virtual Poincar\'e polynomial}(\cite{Mustata2011Lecture8})

The construction has two layers: first the Grothendieck ring, then a ring homomorphism out of it.

\paragraph{Grothendieck ring.} Let $k$ be a field. $K_0(\mathrm{Var}/k)$ is the free abelian group on isomorphism classes $[X]$ of varieties over $k$, modulo
\[
[X] = [Y] + [X \smallsetminus Y]
\]
whenever $Y \subset X$ is closed, with ring structure $[X]\cdot[Y] = [(X\times Y)_{\mathrm{red}}]$. Write $\mathbb{L} := [\mathbb{A}^1_k]$.

\paragraph{Hodge--Deligne polynomial.} For $k=\mathbf{C}$ (or any algebraically closed field of characteristic $0$), Bittner's presentation theorem (via weak factorization) shows that $K_0(\mathrm{Var}/\mathbf{C})$ is generated by classes of smooth connected projective varieties, with relations arising only from blow-ups along smooth centers. This allows one to \emph{define} a ring homomorphism by specifying its value on smooth projective varieties and checking compatibility with blow-ups. Applied with values in $\mathbf{Z}[u,v]$, this produces the unique ring homomorphism
\[
E(-,u,v) : K_0(\mathrm{Var}/\mathbf{C}) \to \mathbf{Z}[u,v]
\]
such that for $X$ smooth projective,
\[
E(X,u,v) = \sum_{p,q=0}^{\dim X} (-1)^{p+q}\, h^{p,q}(X)\, u^p v^q,
\]
where $h^{p,q}(X) = \dim H^q(X,\Omega^p_X)$. (The general definition, for arbitrary, possibly singular, non-proper complex varieties, uses the mixed Hodge structure on compactly-supported singular cohomology $H^*_c(X^{\mathrm{an}},\mathbf{Q})$; the formula above is the special case where that mixed Hodge structure is pure, i.e.\ $X$ smooth projective.)

\smallskip

\paragraph{Virtual Poincar\'e polynomial.}
\[
P_{\mathrm{vir}}(X,t) := E(X,t,t)
\]
This is the specialization $u=v=t$, giving a ring homomorphism $K_0(\mathrm{Var}/\mathbf{C}) \to \mathbf{Z}[t]$. Equivalently  and this is the operational characterization  it is the unique invariant that is
\begin{enumerate}
    \item[(i)] \textbf{additive}: $P_{\mathrm{vir}}(X,t) = P_{\mathrm{vir}}(Y,t) + P_{\mathrm{vir}}(X\smallsetminus Y,t)$ for $Y \subset X$ closed;
    \item[(ii)] \textbf{multiplicative}: $P_{\mathrm{vir}}(X\times Y,t) = P_{\mathrm{vir}}(X,t)\,P_{\mathrm{vir}}(Y,t)$;
    \item[(iii)] equal to the stated formula on smooth projective $X$, which for smooth projective $X$ reduces, by the Hodge decomposition, to
    \[
    P_{\mathrm{vir}}(X,t) = \sum_{i\ge0} (-1)^i \dim_{\mathbf{Q}} H^i(X,\mathbf{Q})\, t^i,
    \]
    with $P_{\mathrm{vir}}(X,1) = \chi_{\mathrm{top}}(X)$.
\end{enumerate}

We also suggest \cite{BrionPeyre2002} for definition of the virtual Poincar\'e polynomial on general (possibly singular, non-proper) complex algebraic varieties, and \cite{Ekedahl2009, Joyce2008Handout, BehrendDhillon2007} for its extension to algebraic stacks.

\subsection{A Harder--Narasimhan Stratification}

Having described the natural Whitney stratification on the moduli space/stack of Gieseker vector bundles $\mathcal M_{X_0}$, we now turn to a different stratification of the moduli problem, one governed not by the singularities but by the (in)stability of the underlying sheaves. This Harder--Narasimhan stratification will be the key tool in the next section for computing the virtual Poincar\'e polynomial of $GVB$, since it allows us to compare $GVB$ with the (better understood) moduli stack $TF$ of torsion-free sheaves via the morphism $\theta$, and to organize the unstable locus of each stack into strata indexed by the numerical type of the Harder--Narasimhan filtration.

We first recall the relevant properties of $\theta$ and its behavior on the two natural stratifications of $GVB$ and $TF$ by singularity depth, and then introduce the Harder--Narasimhan strata on $TF$ and their lifts to $GVB$.

\subsubsection{\textbf{The morphism $\theta\colon GVB\to TF$}}

Let $GVB$ denote the moduli stack of Gieseker vector bundles of rank $2$ and degree $1$, and let $TF$ denote the moduli stack of torsion-free sheaves of rank $2$ and degree $1$. There is a projective birational morphism
\[
\theta\colon GVB \longrightarrow TF.
\]
We record the following well-known features of $\theta$.

\begin{enumerate}
    \item The stack $GVB$ admits a Whitney stratification by singularity depth,
    \[
    GVB_0 := GVB \ \supsetneq\ GVB_1 := \partial(GVB) \ \supsetneq\ GVB_2 := \partial^2(GVB),
    \]
    where $\partial(-)$ denotes the singular locus of the corresponding stack. Moreover, $GVB_2$ is a smooth stack.

    \item Likewise, $TF$ admits a Whitney stratification
    \[
    TF_0 := TF \ \supsetneq\ TF_1 := \partial(TF) \ \supsetneq\ TF_2 := \partial^2(TF),
    \]
    and $TF_2$ is a smooth stack.

    \item The morphism $\theta$ restricts to an isomorphism
    \[
    \theta \colon GVB_0 \setminus GVB_2 \ \xrightarrow{\ \sim\ }\ TF_0 \setminus TF_2,
    \]
    while over the deepest stratum, $\theta \colon GVB_2 \to TF_2$ is a Zariski-locally trivial $\mathbb{P}^3$-bundle. This $\mathbb{P}^3$-bundle admits the following explicit description. First, there is an isomorphism of stacks
    \[
    TF_2 \ \xrightarrow{\ \sim\ }\ \operatorname{Bun}_{2,-1}(\tilde{X}_0), \qquad \mathcal{F} \ \longmapsto\ \pi^*\mathcal{F}/\mathrm{tor}.
    \]
    Let $\mathcal{E}$ denote the universal bundle over $\tilde{X}_0 \times \operatorname{Bun}_{2,-1}(\tilde{X}_0)$, and write $\mathcal{E}_{x_i} := \mathcal{E}|_{x_i \times \operatorname{Bun}_{2,-1}(\tilde{X}_0)}$ for $i = 1,2$. Then
    \[
    GVB_2 \ \cong\ \mathbb{P}\bigl(\operatorname{Hom}(\mathcal{E}_{x_1}, \mathcal{E}_{x_2})\bigr).
    \]
\end{enumerate}

\subsubsection{\textbf{Harder--Narasimhan stratification on $TF$}}

An unstable torsion-free sheaf $\mathcal{F}$ admits a Harder--Narasimhan filtration
\[
0 \ \subset\ \mathcal{F}_1 \ \subset\ \mathcal{F}_0 := \mathcal{F},
\]
such that
\[
\mu(\mathcal{F}_1) := \deg(\mathcal{F}_1) \ >\ \mu(\mathcal{F}_0/\mathcal{F}_1) = \deg(\mathcal{F}_0/\mathcal{F}_1) \ >\ \tfrac{1}{2}.
\]
Write $d_1 := \deg(\mathcal{F}_1)$ and $d_2 := \deg(\mathcal{F}_0/\mathcal{F}_1)$, and denote the corresponding locally closed stratum by $HN_{d_1,d_2}$. There is a morphism
\[
HN_{d_1,d_2} \ \longrightarrow\ TF_{1,d_1} \times TF_{1,d_2},
\]
whose fiber over a pair $(L_1, L_2) \in TF_{1,d_1} \times TF_{1,d_2}$ is isomorphic to the linear stack
\[
\left[\, \frac{\operatorname{Ext}^1(L_2, L_1)}{\operatorname{Hom}(L_2, L_1)} \,\right].
\]
Here we recall that $d_1 \geq 1/2$ and $d_2 = 1 - d_1$.

\subsubsection{\textbf{Dimension of $HN_{d_1,1-d_1}$}}

\medskip

\paragraph{Step 1: Dimension of the fibers $TF_{1,d}$.}

For a rank-$1$ torsion-free sheaf on the nodal curve $X_0$ (arithmetic genus $g$), the endomorphism ring is always $\operatorname{End} = \mathbb{C}$  this holds both for line bundles on $X_0$ and so the stabilizer at every point of $TF_{1,d}$ is exactly $\mathbb{C}^*$. Hence $TF_{1,d}$ is, generically and along its boundary alike, a $\mathbb{G}_m$-gerbe over its coarse space, which is the compactified Jacobian of $X_0$ of dimension $p_a(X_0) = g$.

\medskip

\paragraph{Step 2: Dimension of the fiber $\bigl[\operatorname{Ext}^1(L_2,L_1)/\operatorname{Hom}(L_2,L_1)\bigr]$.}

\begin{itemize}
\item \textsf{Case 1: Both $L_1$ and $L_2$ are locally free:}For line bundles $L_1, L_2$ on $\tilde{X}_0$ of degrees $d_1, d_2$, Serre duality on the nodal curve $X_0$ gives
\[
\operatorname{Ext}^1(L_2,L_1) = H^1\bigl(X_0,\, L_1 \otimes L_2^{-1}\bigr), \qquad \operatorname{Hom}(L_2,L_1) = H^0\bigl(X_0,\, L_1 \otimes L_2^{-1}\bigr).
\]
The line bundle $L_1 \otimes L_2^{-1}$ has degree
\[
d_1 - d_2 \;=\; d_1 - (1-d_1) \;=\; 2d_1 - 1.
\]
By Riemann--Roch on $X_0$,
\[
\chi\bigl(L_1 \otimes L_2^{-1}\bigr) \;=\; h^0 - h^1 \;=\; (2d_1-1) + 1 - g \;=\; 2d_1 - g.
\]
Hence, the dimension of the linear (quotient) stack $\bigl[\operatorname{Ext}^1/\operatorname{Hom}\bigr]$ is
\[
\dim\bigl[\operatorname{Ext}^1(L_2,L_1)/\operatorname{Hom}(L_2,L_1)\bigr] \;=\; h^1 - h^0 \;=\; -\chi\bigl(L_1\otimes L_2^{-1}\bigr) \;=\; g - 2d_1.
\]

This locus is isomorphic to the $\bigl[\operatorname{Ext}^1(L_2,L_1)/\operatorname{Hom}(L_2,L_1)\bigr]$-bundle stack over $(TF_{1,d_1}\setminus Sing(TF_{1,d_1})) \times (TF_{1,1-d_1}\setminus Sing(TF_{1,1-d_1})$. \textbf{Let us denote this locus by $H^0_{d_1, 1-d_1}$.} 

\

\item \textsf{Case 2: exactly one of $L_1$ and $L_2$ is locally free:} Suppose that exactly one of $L_1$ and $L_2$ is locally free and the other is of local type $\mathfrak m$, where $\mathfrak m$ denotes the maximal ideal at the node. Then from \cite[Lemma 2.5]{Bhosle2006}, it follows that
$$\deg(\mathcal Hom(L_2, L_1))=2d_1-1 \qquad\text{and}\qquad \dim(\operatorname{Ext}^1(L_2, L_1))=\dim(H^1(X_0, \mathcal Hom(L_2, L_1))).$$
Therefore, using Riemann--Roch, we get $\dim(\operatorname{Ext}^1(L_2, L_1))-\dim(\operatorname{Hom}(L_2, L_1))=g-2d_1$. This locus is isomorphic to the $\bigl[\operatorname{Ext}^1(L_2,L_1)/\operatorname{Hom}(L_2,L_1)\bigr]$-bundle stack over
\[
\bigl(TF_{1,d_1}\setminus \operatorname{Sing}(TF_{1,d_1})\bigr)\times \operatorname{Sing}(TF_{1,1-d_1}) \ \coprod\ \operatorname{Sing}(TF_{1,d_1})\times\bigl(TF_{1,1-d_1}\setminus \operatorname{Sing}(TF_{1,1-d_1})\bigr).
\]

\

\item \textsf{Case 3: Both $L_1$ and $L_2$ are not locally free:}
Suppose that both $L_1$ and $L_2$ are of local type $m$, where $m$ denotes the maximal ideal at the node. Then from \cite[Lemma 2.5]{Bhosle2006}, it follows that 

$$deg(\mathcal Hom(L_2, L_1))=2d_1$$ and $$dim(Ext^1(L_2, L_1))=dim(H^1(X_0, \mathcal Hom(L_2, L_1)))+2$$

Therefore, using Riemann-Roch, we get $$dim(Ext^1(L_2, L_1))-dim(Hom(L_2, L_1))=g-2d_1+1.$$ This locus is isomorphic to the $\bigl[\operatorname{Ext}^1(L_2,L_1)/\operatorname{Hom}(L_2,L_1)\bigr]$-bundle stack over $Sing(TF_{1,d_1})\times Sing(TF_{1,1-d_1})\cong Pic^{d_1}(\widetilde{X_0})\times Pic^{1-d_1}(\widetilde{X_0})$. \textbf{Let us denote this locus by $H^2_{d_1, 1-d_1}$. }

\subsubsection*{Further stratification of $HN^2$}\label{FurStrat}

Write $HN^2 = HN^{2,0}\sqcup HN^{2,1}\sqcup HN^{2,2}$, according to whether the extension $E$ of $L_2$ by $L_1$ has local type $\mathcal O\oplus\mathcal O$, $\mathcal O\oplus\mathfrak m$, or $\mathfrak m\oplus\mathfrak m$ at the node; $\theta$ restricts to an isomorphism over $HN^{2,0}\sqcup HN^{2,1}$ and to a $\mathbb P^3$-bundle over $HN^{2,2}$.

Fix $(L_1,L_2)\in S_{d_1}\times S_{1-d_1}$. Since $L_1,L_2$ are locally free away from the node, $\mathcal Ext^1(L_2,L_1)$ is a skyscraper there, and the local-to-global spectral sequence (using $H^{\geq2}=0$ on the curve) gives the exact sequence
\[
0 \to H^1\bigl(X_0,\mathcal Hom(L_2,L_1)\bigr) \to \mathrm{Ext}^1_{X_0}(L_2,L_1) \xrightarrow{\ \rho\ } \mathcal Ext^1(L_2,L_1)_{\mathrm{node}}\cong k_\alpha\oplus k_\beta \to 0,
\]
so $\rho$ is precisely the restriction of a global extension class to its local isomorphism type at the node, with kernel $h^1=N-2$, $N=\dim\mathrm{Ext}^1_{X_0}(L_2,L_1)$. It is now straightforward to check that 
\[
HN^{2,0} = \rho^{-1}(\mathbb G_m\times\mathbb G_m), \qquad
HN^{2,1} = \rho^{-1}\bigl((\mathbb G_m\times\{0\})\sqcup(\{0\}\times\mathbb G_m)\bigr), \qquad
HN^{2,2} = \rho^{-1}(0,0).
\]
A linear splitting of $\rho$ trivializes each preimage as $\mathbb A^{h^1}\times(\cdots)$, and dividing by the automorphisms $\mathrm{Hom}(L_2,L_1)\cong\mathbb A^{e_0}$ (a uniform $t^{-2e_0}$) together with Riemann--Roch, $N-e_0=g-2d_1+1$, gives the three fibre classes (i.e., the fibres of the map $HN^{2,i}\to S_{d_1}\times S_{1-d_1}$)
\[
\mathbb P(HN^{2,0}\text{-fibre};t) = (t^2-1)^2\,t^{2(g-2d_1-1)}, \qquad
\mathbb P(HN^{2,1}\text{-fibre};t) = 2(t^2-1)\,t^{2(g-2d_1-1)},
\]
\[
\mathbb P(HN^{2,2}\text{-fibre};t) = t^{2(g-2d_1-1)},
\]
using $\mathbb P(\mathbb G_m\times\{0\};t)=\mathbb P(\{0\}\times\mathbb G_m;t)=t^2-1$ for the two disjoint punctured lines in $HN^{2,1}$, and $\mathbb P(\{(0,0)\};t)=1$ for $HN^{2,2}$. As a check, the three base classes sum to $\mathbb P(\mathbb A^2;t)=t^4$: $(t^2-1)^2+2(t^2-1)+1=t^4$.

Multiplying by $\mathbb P(S_{d_1};t)\,\mathbb P(S_{1-d_1};t)$ for $(L_1,L_2)$ ranging over $S_{d_1}\times S_{1-d_1}$:
\[
\ \begin{aligned}
\mathbb P(HN^{2,0};t) &= t^{2(g-2d_1-1)}(t^2-1)^2\,\mathbb P(S_{d_1};t)\mathbb P(S_{1-d_1};t), \\
\mathbb P(HN^{2,1};t) &= 2\,t^{2(g-2d_1-1)}(t^2-1)\,\mathbb P(S_{d_1};t)\mathbb P(S_{1-d_1};t), \\
\mathbb P(HN^{2,2};t) &= t^{2(g-2d_1-1)}\,\mathbb P(S_{d_1};t)\mathbb P(S_{1-d_1};t).
\end{aligned}\ 
\]
(On lifting to $\widetilde{HN}$: since $\theta$ is an isomorphism over $HN^{2,0}\sqcup HN^{2,1}$, those two classes pass through unchanged; only $\mathbb P(HN^{2,2};t)$ picks up the extra factor of $\mathbb P(\mathbb P^3;t)$ from the $\mathbb P^3$-bundle structure of $\theta$ there.)

\end{itemize}
\

\subsubsection{\textbf{Harder--Narasimhan stratification on $\mathrm{GVB}$}}

The Harder--Narasimhan strata on $\mathrm{TF}$ lift to $\mathrm{GVB}$ via $\theta$: we define
\[
\widetilde{HN}_{d_1,d_2} \ :=\ \theta^{-1}\bigl(HN_{d_1,d_2}\bigr).
\]

Recall that $HN_{d_1,1-d_1}$ is a union of three locally closed substacks,
\[
HN_{d_1,1-d_1} \;=\; HN^0_{d_1,1-d_1}\ \cup\ HN^1_{d_1,1-d_1}\ \cup\ HN^2_{d_1,1-d_1},
\]
and we have seen that the deepest piece splits further according to the local type of the extension $E$ at the node,
\[
HN^2_{d_1,1-d_1} \;=\; HN^{2,0}_{d_1,1-d_1}\ \sqcup\ HN^{2,1}_{d_1,1-d_1}\ \sqcup\ HN^{2,2}_{d_1,1-d_1},
\]
corresponding to local type $\mathcal O\oplus\mathcal O$, $\mathcal O\oplus\mathfrak m$, and $\mathfrak m\oplus\mathfrak m$ respectively. Altogether,
\[
HN_{d_1,1-d_1} \;=\; HN^0_{d_1,1-d_1}\ \sqcup\ HN^1_{d_1,1-d_1}\ \sqcup\ HN^{2,0}_{d_1,1-d_1}\ \sqcup\ HN^{2,1}_{d_1,1-d_1}\ \sqcup\ HN^{2,2}_{d_1,1-d_1}.
\]

This refines the earlier description of $\theta$: it is an isomorphism over the larger locus
\[
HN^0_{d_1,1-d_1}\ \cup\ HN^1_{d_1,1-d_1}\ \cup\ HN^{2,0}_{d_1,1-d_1}\ \cup\ HN^{2,1}_{d_1,1-d_1},
\]
and a $\mathbb P^3$-bundle only over the single remaining piece $HN^{2,2}_{d_1,1-d_1}$  \emph{not} over all of $HN^2_{d_1,1-d_1}$, and \emph{not} over $HN^{2,1}_{d_1,1-d_1}$ either, since the local type $\mathcal O\oplus\mathfrak m$ still leaves $\theta$ an isomorphism. The $\mathbb P^3$-degeneracy occurs precisely at the fully non-locally-free type $\mathfrak m\oplus\mathfrak m$.

\begin{align*}
\mathbb P\bigl(\widetilde{HN}_{d_1,1-d_1};t\bigr) \;=\;& \mathbb P\bigl(HN^0_{d_1,1-d_1};t\bigr) + \mathbb P\bigl(HN^1_{d_1,1-d_1};t\bigr) \\
&+ \mathbb P\bigl(HN^{2,0}_{d_1,1-d_1};t\bigr) + \mathbb P\bigl(HN^{2,1}_{d_1,1-d_1};t\bigr) \\
&+ \mathbb P(\mathbb P^3;t)\cdot\mathbb P\bigl(HN^{2,2}_{d_1,1-d_1};t\bigr),
\end{align*}
with
\[
\mathbb P\bigl(HN^{2,0}_{d_1,1-d_1};t\bigr) = t^{2(g-2d_1-1)}(t^2-1)^2\,\mathbb P(S_{d_1};t)\,\mathbb P(S_{1-d_1};t),
\]
\[
\mathbb P\bigl(HN^{2,1}_{d_1,1-d_1};t\bigr) = 2\,t^{2(g-2d_1-1)}(t^2-1)\,\mathbb P(S_{d_1};t)\,\mathbb P(S_{1-d_1};t),
\]
\[
\mathbb P\bigl(HN^{2,2}_{d_1,1-d_1};t\bigr) = t^{2(g-2d_1-1)}\,\mathbb P(S_{d_1};t)\,\mathbb P(S_{1-d_1};t).
\]
\

\

\subsubsection{\textbf{Computation of the total volume $\mathbb{P}(GVB;t)$ utilizing the Whitney stratification}}

We now compute the virtual Poincar\'e polynomial $\mathbb{P}(GVB;t)$ of the moduli stack $GVB$ of rank-$2$, degree-$1$ Gieseker vector bundles, using the Whitney stratification by degeneration depth,
\[
GVB \;=\; GVB_0 \;\supsetneq\; GVB_1 \;\supsetneq\; GVB_2 .
\]
By additivity of the virtual Poincar\'e polynomial under scissor relations,
\[
\mathbb{P}(GVB;t) \;=\; \mathbb{P}(GVB_0\setminus GVB_1;t) \;+\; \mathbb{P}(GVB_1\setminus GVB_2;t) \;+\; \mathbb{P}(GVB_2;t).
\]
We compute each of the three contributions in turn, using the fibration structures over moduli spaces on the normalization $\tilde X_0$.

\paragraph{1. The open stratum $GVB_0\setminus GVB_1$.}
This stratum parametrizes stable Gieseker bundles with no rational chain, i.e.\ ordinary stable bundles on the nodal curve $X_0$. It is a $\mathrm{GL}_2$-fibration over $\Bun_{\tilde X_0}(2,1)$, the moduli stack of rank-$2$, degree-$1$ vector bundles on the smooth normalization $\tilde X_0$: the fiber $\mathrm{GL}_2$ accounts for the gluing data at the two preimages $x^+,x^-$ of the node via the generalized parabolic structure. Hence
\[
\mathbb{P}(GVB_0\setminus GVB_1;t) \;=\; \mathbb{P}(\Bun_{\tilde X_0}(2,1);t)\cdot \mathbb{P}(\mathrm{GL}_2;t).
\]

\paragraph{2. The middle stratum $GVB_1\setminus GVB_2$.}
This stratum corresponds to Gieseker bundles with a rational chain of positive but minimal length (depth-$1$ degenerations), and is realized as a $\mathbb{C}^*$-bundle over $PBun_{\tilde X_0}$, the moduli stack of parabolic bundles on $\tilde X_0$.

To describe this explicitly, let $\mathbb E$ denote the universal bundle on $\tilde X_0 \times \Bun_{\tilde X_0}$, and let $\mathbb E_{x^+}$, $\mathbb E_{x^-}$ denote its restrictions along $x^+\times\Bun_{\tilde X_0}$ and $x^-\times\Bun_{\tilde X_0}$, respectively. Write $PBun_{\tilde X_0,x^+,x^-}$ for the moduli stack of quasi-parabolic bundles on $\tilde X_0$, with tautological parabolic subbundles $\mathbb L_{x^+}$ and $\mathbb L_{x^-}$. Consider the $\mathbb{C}^*$-bundle
\[
\operatorname{Isom}\bigl(\mathbb E_{x^+}/\mathbb L_{x^+},\ \mathbb L_{x^-}\bigr) \;\longrightarrow\; PBun_{\tilde X_0,x^+,x^-}.
\]
Then
\[
GVB_1\setminus GVB_2 \;\cong\; \operatorname{Isom}\bigl(\mathbb E_{x^+}/\mathbb L_{x^+},\ \mathbb L_{x^-}\bigr),
\]
and consequently
\[
\mathbb{P}(GVB_1\setminus GVB_2;t) \;=\; \mathbb{P}(PBun_{\tilde X_0};t)\cdot \mathbb{P}(\mathbb{C}^*;t) \;.
\]

\paragraph{3. The deepest stratum $GVB_2$.}
The closed stratum $GVB_2$, corresponding to maximal degeneration, is a
Zariski-locally trivial $\mathbb{P}^3$-bundle over $\Bun_{\tilde X_0}(2,-1)$,
the moduli stack of rank-$2$, degree-$(-1)$ bundles on the normalization
$\tilde X_0$. Indeed, by the description of $\theta$ above,
\[
GVB_2 \;\cong\; \mathbb{P}\bigl(\operatorname{Hom}(\mathcal{E}_{x^+},\mathcal{E}_{x^-})\bigr)
\;\longrightarrow\; TF_2 \;\cong\; \Bun_{\tilde X_0}(2,-1),
\]
where $\mathcal{E}$ is the universal bundle on $\tilde X_0\times\Bun_{\tilde X_0}(2,-1)$.
By multiplicativity for fibrations,
\[
\mathbb{P}(GVB_2;t) \;=\; \mathbb{P}(\mathbb{P}^3;t)\cdot\mathbb{P}(\Bun_{\tilde X_0}(2,-1);t).
\]

\paragraph{Combining the strata.}
Assembling the three contributions via the scissor relation above yields

\begin{equation}
\begin{aligned}
\mathbb{P}(GVB;t) \;=\;& \underbrace{\mathbb{P}(\Bun_{\tilde X_0}(2,1);t)\,\mathbb{P}(\mathrm{GL}_2;t)}_{GVB_0\setminus GVB_1}
\;+\; \underbrace{\mathbb{P}(\mathbb{C}^*;t)\,\mathbb{P}(PBun_{\tilde X_0};t)}_{GVB_1\setminus GVB_2} \\
&+\; \underbrace{\mathbb{P}(\mathbb{P}^3;t)\,\mathbb{P}(\Bun_{\tilde X_0}(2,-1);t)}_{GVB_2}.
\end{aligned}
\end{equation}

On the other hand, the virtual Poincar\'e polynomial also satisfies a recursive formula coming from the Harder--Narasimhan stratification, via scissor relations together with fibration multiplicativity:

\begin{equation}
\mathbb{P}(GVB;t) \;=\; \mathbb{P}(GVB^{\mathrm{ss}};t) \;+\; \sum_{\substack{d_1+d_2=1 \\ d_1>1/2}} \mathbb{P}\bigl(\widetilde{HN}_{d_1,1-d_1};t\bigr).
\end{equation}

where $GVB^{\mathrm{ss}}$ denotes the open substack of Gieseker-semistable bundles, whose coarse moduli space is the main component of $\mathcal{M}_{X_0}$.

Combining this with the formula for $\mathbb{P}(GVB;t)$ obtained from the Whitney stratification, together with the Harder--Narasimhan recursion,
\[
\mathbb{P}(GVB;t) \;=\; \mathbb{P}(GVB^{\mathrm{ss}};t) \;+\; \sum_{\substack{d_1+d_2=1\\ d_1>1/2}} \mathbb{P}\bigl(\widetilde{HN}_{d_1,1-d_1};t\bigr),
\]
we obtain the final formula for the virtual Poincar\'e polynomial of the semistable locus:

\smallskip

\begin{theorem}[Virtual Poincaré polynomial]
\label{thm:vpp}
Let $\mathrm{GVB}$ denote the moduli stack of Gieseker vector bundles of rank $2$ and degree $1$ on $X_0$, and write $g:=g(X_0)$. With $\mathbb{P}(A_d;t)=(1+t)^{2(g-1)}$ and $\mathbb{P}(S_d;t)=(1+t)^{2(g-1)}/(t^2-1)$, the virtual Poincaré polynomial of the Gieseker-semistable locus obtained from the Whitney stratification is
\begin{equation}
\mathbb{P}(GVB^{\mathrm{ss}};t)
= \frac{t^{2}(1+t)^{4g-7}\Bigl[(t^{4}+1)(1-t+t^{2})^{2(g-1)} - 2t^{2g}\Bigr]}{(t-1)^{3}}.
\label{eq:gvbss-v2}
\end{equation}
and
\begin{equation}
\mathbb{P}(\mathcal{M}_{X_0},t)
=(t^2-1)\,\mathbb{P}(GVB^{\mathrm{ss}};t)
= \frac{t^{2}(1+t)^{4g-6}\Bigl[(t^{4}+1)(1-t+t^{2})^{2(g-1)} - 2t^{2g}\Bigr]}{(t-1)^{2}}.
\label{eq:coarse-v2}
\end{equation}
\end{theorem}

\begin{example}
For $g=2$,
\[
\mathbb{P}(\mathcal{M}_{X_0},t)\big|_{g=2} \;=\; t^{10}+2t^{9}+3t^{8}+6t^{7}+8t^{6}+6t^{5}+3t^{4}+2t^{3}+t^{2}.
\]
For $g=3$,
\begin{align*}
\mathbb{P}(\mathcal{M}_{X_0},t)\big|_{g=3} \;=\;& t^{18}+4t^{17}+8t^{16}+16t^{15}+33t^{14}+56t^{13}+84t^{12}+116t^{11}\\
&+132t^{10}+116t^{9}+84t^{8}+56t^{7}+33t^{6}+16t^{5}+8t^{4}+4t^{3}+t^{2}.
\end{align*}
\end{example}
\begin{proof}
The class of the full stack of rank-$2$, degree-$1$ bundles on $\tilde{X}_{0}$ is
\begin{equation}
\mathbb{P}(\mathrm{Bun}_{\tilde{X}_{0}}(2,1);t)
=\frac{(1+t)^{2(g-1)}(1+t^{3})^{2(g-1)}}{(t^{2}-1)^{2}(t^{4}-1)}
= \frac{(1+t)^{4g-7}(1-t+t^{2})^{2(g-1)}}{(t-1)^{3}(1+t^{2})}.
\label{eq:Bun-full}
\end{equation}
The Whitney stratification gives
\[
\mathbb{P}(GVB;t)
=\mathbb{P}(\mathrm{Bun}_{\tilde{X}_{0}}(2,1);t)\cdot(t^{8}+t^{6}+t^{4}+t^{2})
= \frac{t^2(t^4+1)(1+t)^{4g-7}(1-t+t^2)^{2(g-1)}}{(t-1)^3},
\]
the $(1+t^2)$ in \eqref{eq:Bun-full} cancelling against the factor $t^8+t^6+t^4+t^2=t^2(t^2+1)(t^4+1)$.

Write $A=(1+t)^{2(g-1)}$ and $S=A/(t^{2}-1)$. With $\theta$ an isomorphism on $HN^0\cup HN^1\cup HN^{2,0}\cup HN^{2,1}$ and a $\mathbb P^3$-bundle only on $HN^{2,2}$ (the fully non-locally-free type $\mathfrak m\oplus\mathfrak m$ at the node; see the corrected stratification above),
\begin{align*}
\mathbb{P}(\widetilde{HN}^{0};t) &=t^{2(g-2d_{1})}A^{2},\\
\mathbb{P}(\widetilde{HN}^{1};t) &=2t^{2(g-2d_{1})}AS,\\
\mathbb{P}(\widetilde{HN}^{2,0};t) &=t^{2(g-2d_{1}-1)}(t^{2}-1)^{2}S^{2}=t^{2(g-2d_1-1)}A^2,\\
\mathbb{P}(\widetilde{HN}^{2,1};t) &=2t^{2(g-2d_1-1)}(t^2-1)S^2 = 2t^{2(g-2d_1-1)}AS,\\
\mathbb{P}(\widetilde{HN}^{2,2};t) &=\mathbb{P}(\mathbb{P}^{3};t)\,t^{2(g-2d_{1}-1)}S^{2}.
\end{align*}
Using $A=(t^2-1)S$ to write everything in terms of $S$, these five terms sum to
\begin{equation}
\mathbb{P}(\widetilde{HN}_{d_{1},1-d_{1}};t)
= 2(t^2+1)\,t^{2(g-2d_1+1)}\,S^2
= \frac{2(t^{2}+1)(1+t)^{4g-6}\,t^{2(g-2d_{1}+1)}}{(t-1)^{2}}.
\label{eq:HN-explicit-v2}
\end{equation}
The geometric series $\sum_{d_{1}\ge 1}t^{2(g-2d_{1}+1)}=t^{2}\cdot t^{2g}/(t^{4}-1)$ and $t^{4}-1=(t-1)(t+1)(1+t^{2})$ give, after the $(1+t^2)$ cancels against the numerator factor $(t^2+1)$,
\begin{equation}
\sum_{d_{1}\ge 1}\mathbb{P}(\widetilde{HN}_{d_{1},1-d_{1}};t)
= \frac{2t^{2g+2}(1+t)^{4g-7}}{(t-1)^{3}}.
\label{eq:HN-sum-v2}
\end{equation}
Subtracting \eqref{eq:HN-sum-v2} from $\mathbb{P}(GVB;t)$ gives \eqref{eq:gvbss-v2}.
\end{proof}

\

\subsection{An alternate method}
In the previous section, we computed $\mathbb P(GVB^{\mathrm{ss}};t)$ directly, applying the Harder--Narasimhan recursion to the Whitney-stratified stack $GVB$ itself. That method has its imporatance because it shows the nature of the complicated Harder Narasimhan stratification of the moduli stack of Gieseker vector bundles. 

As a cross-check, we now compute it by an independent route: first compute $P(TF^{\mathrm{ss}}_{2,1},t)$ on its own terms, running the same Harder--Narasimhan argument on $TF_{2,1}$, and then only pass to $GVB^{\mathrm{ss}}$ via the comparison map $\theta\colon GVB\to TF_{2,1}$ (an isomorphism away from the deepest stratum, a $\mathbb P^3$-bundle on it). Since $TF_{2,1}$'s own stratification carries no $\mathbb P^3$-twist internally, this computation is structurally simpler, and comparing against the previous section after reassembling via $\theta$ gives an independent check on both derivations.

Throughout, we keep every quantity symbolic and substitute closed forms only once, at the very end of \S\ref{sec:assembling}.

\paragraph{Notation.} Set
\[
B(t):=P(\mathrm{Bun}_{2,1},t),\qquad
S(t):=P(S_{d_1},t),\qquad
T(t):=P(TF_{1,d_1},t),
\]
$S(t)$ being independent of $d_1$. Unlike in an earlier draft, $T(t)$ is \emph{not} $(1+t+t^2)S(t)$: as established above, $TF_{1,d_1}$ decomposes as $(TF_{1,d_1}\setminus S_{d_1})\sqcup S_{d_1}$, with
\[
P(TF_{1,d_1}\setminus S_{d_1},t) = (1+t)^{2(g-1)}, \qquad S(t) = \frac{(1+t)^{2(g-1)}}{t^2-1},
\]
(the first since $\mathrm{Pic}^{d_1}(X_0)$, a $\mathbb G_m$-bundle over $\mathrm{Pic}^{d_1}(\tilde X_0)$ of class $(1+t)^{2(g-1)}(t^2-1)$, is itself a trivial $\mathbb G_m$-gerbe, so the two factors of $(t^2-1)$ cancel), so that
\[
T(t) = (1+t)^{2(g-1)} + \frac{(1+t)^{2(g-1)}}{t^2-1} = \frac{t^2(1+t)^{2(g-1)}}{t^2-1} = t^2\,S(t).
\]
Let
\[
\Pi(t):=P(\mathrm{GL}_2,t)+P(\mathbb P^1,t)^2P(\mathbb C^*,t)+1,\qquad
C(t):=P(\mathbb P^3,t)-1
\]
be the two stack-correction factors contributed by, respectively, the HN-semistable locus and the fiber of $\theta$ in the deepest stratum.

By the Harder--Narasimhan recursion on $TF_{2,1}$,
\begin{equation}
P(TF^{ss}_{2,1},t) = \Pi(t)\,B(t) \;-\;\sum_{d_1\ge1} t^{2(g-2d_1)}\Bigl[T(t)^2+(t^2-1)S(t)^2\Bigr]. \label{eq:alt-combined}
\end{equation}
With $T(t)=t^2S(t)$, the correction bracket collapses to a multiple of $S(t)^2$:
\[
T(t)^2+(t^2-1)S(t)^2 = Q(t)\,S(t)^2,\qquad Q(t):=t^4+t^2-1,
\]
and the sum over $d_1$ is the geometric series
\[
\Gamma(t,g):=\sum_{d_1\ge1} t^{2(g-2d_1)} = \frac{t^{2g}}{t^4-1}.
\]
Combining these gives the symbolic closed form
\begin{equation}
P(TF^{ss}_{2,1},t) = \Pi(t)\,B(t) \;-\; \Gamma(t,g)\,Q(t)\,S(t)^2. \label{eq:alt-symbolic}
\end{equation}

\paragraph{Closed forms (substituted only here).}
Inserting
\begin{align*}
\Pi(t) &= t^8, \\
C(t) &= t^2+t^4+t^6, \\
Q(t) &= t^4+t^2-1, \\
B(t) &= \frac{(1+t)^{2(g-1)}(1+t^3)^{2(g-1)}}{(t^2-1)^2(t^4-1)}, \\
S(t) &= \frac{(1+t)^{2(g-1)}}{t^2-1},
\end{align*}
and writing $1+t^3=(1+t)(1-t+t^2)$, both terms have denominator $(t^2-1)^2(t^4-1)$. The numerator is
\[
(1+t)^{4(g-1)}\Bigl[t^8(1-t+t^2)^{2(g-1)}-t^{2g}(t^4+t^2-1)\Bigr].
\]
Since $(t^2-1)^2(t^4-1)=(t-1)^3(t+1)^3(1+t^2)$, clearing the factor $(t+1)^3$ gives
\begin{equation}
\,P(TF^{ss}_{2,1},t) = \frac{(1+t)^{4g-7}\Bigl[t^8(1-t+t^2)^{2(g-1)}-t^{2g}\bigl(t^4+t^2-1\bigr)\Bigr]}{(t-1)^3(1+t^2)}\,\label{eq:alt-final}
\end{equation}
or, equivalently,
\begin{equation}
P(TF^{ss}_{2,1},t) = \frac{(1+t)^{4g-7}\Bigl[t^{2g}\bigl(t^4+t^2-1\bigr)-t^8(1-t+t^2)^{2(g-1)}\Bigr]}{(1-t)^3(1+t^2)}.
\end{equation}

Since $\gcd(2,1)=1$, the semistable stack is a $\mathbb{G}_m$-gerbe over its coarse moduli scheme, and the virtual Poincaré polynomial of the scheme is obtained by multiplying by the class of $\mathbb{C}^*$:
\begin{equation}
\mathbb{P}(\mathcal{M}_{\mathrm{TF}},t)
=(t^2-1)\,P(TF^{ss}_{2,1},t)
=\frac{-(1+t)^{4g-6}\Bigl[t^{2g}(t^4+t^2-1)-t^8(1-t+t^2)^{2(g-1)}\Bigr]}{(1-t)^2(1+t^2)}.
\label{eq:TF-scheme}
\end{equation}

For $g=2$,
\begin{equation}
\mathbb{P}(\mathcal{M}_{\mathrm{TF}},t)
=t^{10}+2t^9+2t^8+4t^7+6t^6+4t^5+t^4.
\end{equation}
For $g=3$,
\begin{equation}
\mathbb{P}(\mathcal{M}_{\mathrm{TF}},t)
=t^{18}+4t^{17}+7t^{16}+12t^{15}+25t^{14}+40t^{13}+52t^{12}+68t^{11}+77t^{10}+60t^9+29t^8+8t^7+t^6.
\end{equation}

\smallskip

\subsection{Assembling \texorpdfstring{$P(GVB^{\mathrm{ss}},t)$}{P(GVB\^{}ss,t)}}\label{sec:assembling}
The same scissor relation for $\theta$ gives, on the deepest stratum,
\[
P(TF_2^{ss},t) = P(\mathrm{Bun}^{ss}_{2,-1}, t),
\]
and hence
\begin{equation}
P(GVB^{ss},t)
= P(TF^{ss}_{2,1},t)
+ P(TF_2^{ss},t)\bigl(P(\mathbb{P}^3, t)-1\bigr).
\label{eq:final-step-symbolic}
\end{equation}
Tensoring by a line bundle of degree $-1$ identifies the deepest stratum with $\mathrm{Bun}^{ss}_{2,1}$, and $P(\mathbb{P}^3,t)-1=t^2+t^4+t^6$, so
\begin{equation}
\begin{aligned}
P(GVB^{ss},t)
&=\frac{(1+t)^{4g-7}\Bigl[t^{2g}(t^4+t^2-1)-t^8(1-t+t^2)^{2(g-1)}\Bigr]}{(1-t)^3(1+t^2)}\\
&\qquad+\frac{(t^2+t^4+t^6)\,(1+t)^{2(g-1)}\Bigl((1+t^3)^{2(g-1)}-t^{2(g-1)}(1+t)^{2(g-1)}\Bigr)}{(t^2-1)^2(t^4-1)}.
\end{aligned}
\end{equation}
The coarse moduli scheme is obtained by multiplying by the class of $\mathbb{C}^*$:
\begin{equation}
P(\mathcal{M}_{X_0},t)=(t^2-1)\,P(GVB^{ss},t).
\label{eq:coarse-v2}
\end{equation}
The denominator cancels, and the result is a polynomial of degree $2(4g-3)$ with leading coefficient $1$.

For $g=2$,
\begin{equation}
P(\mathcal{M}_{X_0},t)
=t^{10}+2t^9+3t^8+6t^7+8t^6+6t^5+3t^4+2t^3+t^2.
\end{equation}
For $g=3$,
\begin{equation}
\begin{aligned}
P(\mathcal{M}_{X_0},t)
&=t^{18}+4t^{17}+8t^{16}+16t^{15}+33t^{14}+56t^{13}\\
&\qquad+84t^{12}+116t^{11}+132t^{10}+116t^{9}+84t^{8}\\
&\qquad+56t^{7}+33t^{6}+16t^{5}+8t^{4}+4t^{3}+t^{2}.
\end{aligned}
\end{equation}

\begin{example}[A singular projective variety with vanishing constant term]
The virtual Poincaré polynomial is additive on strata, but it is not the ordinary Poincaré polynomial once the variety is singular. In particular, connectedness and projectivity do not force the constant term to be \(1\): a stratum isomorphic to a non-proper variety may contribute \(0\) or \(-1\), and these contributions need not cancel to \(1\) in the total. The nodal cubic is the standard illustration. 

Let \(C\subset\mathbb{P}^2\) be the nodal cubic
\[
y^2z=x^3+x^2z,
\]
with node \(p=[0:0:1]\). Thus \([C]=[\mathbb{P}^1]-[\mathrm{pt}]\) in the Grothendieck ring of varieties. Decompose \(C\) into the node and its complement. The complement is isomorphic to \(\mathbb{C}^*\), and
\[
\mathbb{P}(\mathbb{P}^1;t)=1+t^2
\]
gives
\[
\mathbb{P}(\mathbb{C}^*;t)
=\mathbb{P}(\mathbb{P}^1;t)-2\mathbb{P}(\mathrm{pt};t)
=(1+t^2)-2
=t^2-1.
\]
Additivity therefore yields
\[
\mathbb{P}(C;t)
=\mathbb{P}(C\setminus\{p\};t)+\mathbb{P}(\{p\};t)
=(t^2-1)+1
=t^2.
\]
The constant term vanishes. A similar cancellation is what occurs for the total stack above in our theorem on virtual Poincaré Polynomial.
\end{example}
\

\

\section{Intersection Poincar\'e Polynomial of $GVB_{(2,1)}$}

\subsubsection{Setup}
Let $\mathcal M_{X_0}$ be the Gieseker moduli space of stable rank-2, degree-1 bundles on $X_0$, of dimension $n$, with Whitney stratification
\[
\mathcal M_{X_0} = \overline{M^0} \supset \overline{M^1} \supset M^2,
\]
where $M^0$ is the open locus of stable bundles on $X_0$ itself, $\overline{M^1}$ (codimension $1$) is the locus of stable Gieseker bundles on the length-one expanded curve, and $M^2$ (codimension $2$) on the length-two expanded curve.

Before proceeding, we record the provenance of the stack $\mathrm{GVBD}$ used below. It is
the degeneration stack constructed by Kausz \cite{Kausz2005}, generalizing Gieseker's
original rank-$2$ construction \cite{Gieseker1984} to arbitrary rank: for a family of smooth
curves degenerating to a curve with a single node, Kausz builds a regular, flat moduli stack
$\mathrm{GVB}_n(C/B)$ of \emph{Gieseker vector bundles} of rank $n$ on the rational-bridge
modifications of the special fibre, exactly as in \S2.0.3--\S2.0.4 above; we write
$\mathrm{GVBD}$ for the normalisation of the closed fibre of the degeneration $\mathrm{GVB}_n(C/B)$. The elements of this stack are called Gieseker vector bundle data. The main structural theorem of
\cite{Kausz2005} identifies the stack $\mathrm{GVBD}$ with a Zariski-locally trivial
$KGL_n$-fibration over $\Bun_{\tilde X_0}(n,d)$, where $KGL_n$ is Kausz's canonical modular
compactification of $GL_n$ from \cite{Kausz2000}.

For $n=2$, $KGL_2$ is obtained in \cite{Kausz2000} by exactly \emph{two} successive
blow-ups of the Grassmanian compactification of $GL_2$. This is the same two-step
blow-up we carried out by hand in \S2 in passing from
$\operatorname{GPB}_{\tilde X_0}=\operatorname{Gr}_2(\mathbb E_{x^+}\oplus\mathbb E_{x^-})$
to $\widetilde{\operatorname{GPB}}_{\tilde X_0}$, blowing up $W_1$ and then $W_2$: $GL_2$
sits inside $\operatorname{Gr}_2(\mathbb C^2\oplus\mathbb C^2)$ as the locus of subspaces
that are graphs of isomorphisms, and the two loci Kausz blows up correspond, fibrewise, to
the rank-$0$ vertices $W_1,W_2$ of the boundary cones $D_1,D_2$ resolved in \S2. 

Under this identification, $\mathrm{GVBD}$ coincides with the moduli stack of
\emph{$H$-bundles} considered by Seshadri in the ICTP notes \cite{ICTP}: a Gieseker vector
bundle datum (GVBD) in the sense of \S2.0.4 is precisely an $H$-bundle, and in the
rank-and-degree-coprime case the two notions of (semi)stability agree, both being calibrated
so that the pushforward to $X_0$ is torsion-free of the correct rank and degree. This notion
of semistability  for Gieseker vector bundle data, equivalently for $H$-bundles  is
studied thoroughly by Sun \cite{Sun2008} in the coprime case, who constructs the resulting
coarse moduli space by treating the problem as one of variation of GIT: as the linearization
is varied, the (semi)stable locus changes only across finitely many walls, and the resulting
chambers are exactly the spaces $\mathcal M^{\alpha}_{\mathrm{GVBD}}$, $\alpha\in[0,1]$, of
our flip-flop diagram in \S4.4, with $\alpha=1$ the chamber whose coarse space is the blow up of the moduli space of GPBs $\mathcal M_{\mathrm{GPB}}$. 

\subsubsection{Setup}
Let $\mathcal M_{X_0}$ be the Gieseker moduli space of stable rank-2, degree-1 bundles on $X_0$, of dimension $n$, with Whitney stratification
\[
\mathcal M_{X_0} = \overline{M^0} \supset \overline{M^1} \supset M^2,
\]
where $M^0$ is the open locus of stable bundles on $X_0$ itself, $\overline{M^1}$ (codimension $1$) is the locus of stable Gieseker bundles on the length-one expanded curve, and $M^2$ (codimension $2$) on the length-two expanded curve. Let $\pi\colon Y\to \mathcal M_{X_0}$ be the normalization $Y := \mathcal M^{\alpha=1}_{\mathrm{GVBD}}$ (the map $B_2$ in the flip-flop diagram), identified as in \S4.1 above with $\widetilde{\mathcal M_{\operatorname{GPB}}}$. As established in \S2, $\pi^{-1}(\overline{M^1}) = \overline{D_1}\cup\overline{D_2}\cup E_1\cup E_2$ has \emph{four} irreducible components, each smooth of dimension $n-1$, meeting in the connected chain $E_1\cup\overline{D_1}\cup\overline{D_2}\cup E_2$: the three consecutive intersections $E_1\cap\overline{D_1}$, $\overline{D_1}\cap\overline{D_2}$, and $\overline{D_2}\cap E_2$ are each isomorphic to $M^2$ (via the identification with $\operatorname{PBun}_{2,1}(\tilde X_0;x^+,x^-)$ from \S2), while the remaining pairs are disjoint, $E_1\cap E_2=\varnothing$, $E_1\cap\overline{D_2}=\varnothing$, $\overline{D_1}\cap E_2=\varnothing$; all of $\pi^{-1}(\overline{M^1})$ is identified with $\overline{M^1}$, and each of the three overlap loci with $M^2$, by $\pi$.

\subsubsection{$\pi$ is small, so $IH^*(\mathcal M_{X_0})$ needs no correction terms}

Recall that a proper, surjective, generically finite morphism $f\colon Y\to X$ is \emph{small} \cite{GoreskyMacPherson1983} if
\[
\dim\{x\in X : \dim f^{-1}(x)\ge r\} < \dim X - 2r \qquad \text{for every } r\ge 1.
\]

Since $\pi\colon Y\to \mathcal M_{X_0}$ is the normalization map, every fibre of $\pi$ is $0$-dimensional, so $\{x : \dim\pi^{-1}(x)\ge 1\} = \emptyset$ for every $r\ge 1$, and the smallness inequality holds vacuously. (This uses only that $\pi$ is generically finite and proper  the count of irreducible components in $\pi^{-1}(\overline{M^1})$ above plays no role here.)

We recall the following theorem. 

\begin{theorem} \cite[\S 6.2 Small maps and Resolutions]{GoreskyMacPherson1983}\label{thm:small}
There is a canonical isomorphism $R\pi_*IC_Y \cong IC_{\mathcal M_{X_0}}$ in $D^b_c(\mathcal M_{X_0},\mathbb Q)$, with no other summands. Since $Y$ is smooth, $IC_Y = \mathbb{Q}_Y[n]$, so $R\pi_*\mathbb{Q}_Y[n] \cong IC_{\mathcal M_{X_0}}$.
\end{theorem}

\begin{corollary}[Intersection Poincar\'e polynomial of $\mathcal M_{X_0}$]\label{cor:main}
\[
P_{IH}(\mathcal M_{X_0}, t) = P(Y, t),
\]
with no correction terms involving $\overline{D_1}$, $\overline{D_2}$, $E_1$, $E_2$, or $M^2$.
\end{corollary}

\begin{proposition}
Let $\mathcal M_{X_0}$ be the Gieseker moduli space of stable rank-2, degree-1 bundles on the irreducible nodal curve $X_0$, and write $g:=g(X_0)$ for its arithmetic genus, so $g(\tilde X_0)=g-1$. Let $Y = \mathcal M^{\alpha=1}_{\mathrm{GVBD}}$ be its smooth resolution via the top GIT chamber, with $\pi\colon Y\to \mathcal M_{X_0}$ the normalisation morphism. Then, for all perversities,
\[
P_{IH}(\mathcal M_{X_0}, t) = P(Y,t) = P(\mathcal M_{\mathrm{VB}},t)\cdot P(\mathrm{KGL}_2,t) + P(\Delta_2,t) - P(\Delta_1,t),
\]
with $P(\mathcal M_{\mathrm{VB}},t)$, $P(\mathrm{KGL}_2,t)$, $P(\Delta_1,t)$, $P(\Delta_2,t)$ as computed below.
\end{proposition}

\subsection{Poincar\'e polynomial of $Y$, via variation of GIT}

We compute $P(Y,t)$ for the moduli spaces of generalized parabolic bundles (GPBs) and Gieseker vector bundles (GVBs) of rank $2$, degree $1$, built from the smooth normalization $\tilde X_0$, which has genus $g-1$ (recall $g:=g(X_0)$).

Semistability for these objects depends on a real parameter $\alpha\in[0,1]$; as $\alpha$ varies, the coarse moduli spaces undergo a sequence of birational transformations (flip-flops): (for details refer to \cite[Theorem 6.2]{ICTP}, we will borrow the notations from this note)
with $Y := \mathcal{M}^{\alpha=1}_{\mathrm{GVBD}}$. We use the following facts (see \cite{ICTP, Gieseker1984}):

\begin{itemize}
\item $\operatorname{Cent}(B_i)\cap\Delta_i = \emptyset$ for $i=1,2$, where $\operatorname{Cent}(B_i)$ denotes the center of the blow-up $B_i$.
\item Writing $J_0,J_1$ for the Picard varieties of degree $0$ and degree $1$ on $\tilde X_0$, there exist vector bundles $F_1,F_2$ over $J_0\times J_1$ of ranks $g-1$ and $g+1$ respectively, with $\mathbb{P}(F_i)\cong \Delta_i$.
\item The blow-up loci and exceptional divisors fit into
\[
\begin{tikzcd}
& E := \Delta_1 \times \Delta_2 \\
\Delta_1 \arrow[ur] & & \Delta_2 \arrow[ul]
\end{tikzcd}
\]
\end{itemize}

\subsubsection{Explicit Poincar\'e polynomials}

The flip-flop diagram gives the recursive relations
\[
P(\mathcal{M}^{\alpha=1}_{\mathrm{GPB}},t) = P(\mathcal{M}^{\alpha=0}_{\mathrm{GPB}},t) + P(\Delta_2,t) - P(\Delta_1,t), \qquad
P(Y,t) = P(\mathcal{M}^{\alpha=0}_{\mathrm{GVBD}},t) + P(\Delta_2,t) - P(\Delta_1,t),
\]
where the open-stratum spaces fiber over the moduli space of vector bundles:
\[
P(\mathcal{M}^{\alpha=0}_{\mathrm{GPB}},t) = P(\mathcal{M}_{\mathrm{VB}},t)\cdot P(\mathrm{Gr}(2,4),t), \qquad
P(\mathcal{M}^{\alpha=0}_{\mathrm{GVBD}},t) = P(\mathcal{M}_{\mathrm{VB}},t)\cdot P(\mathrm{KGL}_2,t),
\]
with $\mathrm{KGL}_2$ the kernel of the natural map $\mathrm{GL}_2\to \mathrm{Gr}(2,4)$.

The standard Poincar\'e polynomials needed are
\[
P(\mathbb{P}^n,t) = 1+t^2+\cdots+t^{2n}, \qquad
P(\mathrm{Gr}(2,4),t) = 1+t^2+2t^4+t^6+t^8,
\]
\[
P(\mathrm{KGL}_2,t) = P(\mathrm{Gr}(2,4),t) + 2P(\mathbb{P}^3,t) - 2P(\mathrm{pt},t) = 1+3t^2+4t^4+3t^6+t^8.
\]
The Poincar\'e polynomial of the moduli space of (varying-determinant) stable rank-2 bundles on the genus-$(g-1)$ curve $\tilde X_0$ is
\[
P(\mathcal{M}_{\mathrm{VB}},t) = \frac{(1+t)^{2(g-1)}\bigl((1+t^3)^{2(g-1)} - t^{2(g-1)}(1+t)^{2(g-1)}\bigr)}{(1-t^2)(1-t^4)},
\]
and the Picard varieties satisfy $P(J_0,t) = P(J_1,t) = (1+t)^{2(g-1)}$.

Using $\mathbb{P}(F_1)\cong\Delta_1$, $\mathbb{P}(F_2)\cong\Delta_2$ of relative dimensions $g-2$ and $g$ over $J_0\times J_1$:
\[
P(\Delta_1,t) = \Bigl(\sum_{k=0}^{g-2}t^{2k}\Bigr)(1+t)^{4(g-1)}, \qquad
P(\Delta_2,t) = \Bigl(\sum_{k=0}^{g}t^{2k}\Bigr)(1+t)^{4(g-1)}.
\]

Combining these,
\begin{multline*}
P(Y,t) = \frac{(1+t)^{2(g-1)}\bigl((1+t^3)^{2(g-1)}-t^{2(g-1)}(1+t)^{2(g-1)}\bigr)}{(1-t^2)(1-t^4)}\cdot\bigl(1+3t^2+4t^4+3t^6+t^8\bigr) \\
+ \Bigl(\sum_{k=0}^{g}t^{2k}\Bigr)(1+t)^{4(g-1)} - \Bigl(\sum_{k=0}^{g-2}t^{2k}\Bigr)(1+t)^{4(g-1)}.
\end{multline*}

\begin{theorem}[Intersection Poincar\'e Polynomial of $\mathcal M_{X_0}$, closed form]\label{thm:IH-closed-form}
Let $\mathcal M_{X_0}$ be the Gieseker moduli space of stable rank-2, degree-1 bundles on the irreducible nodal curve $X_0$, and write $g := g(X_0)$ for its arithmetic genus, so that $g(\tilde X_0) = g-1$. Then
\[
\resizebox{0.98\linewidth}{!}{$
\displaystyle
P_{IH}(\mathcal M_{X_0}, t)
= \frac{(1+t)^{2(g-1)}\bigl((1+t^3)^{2(g-1)}-t^{2(g-1)}(1+t)^{2(g-1)}\bigr)}{(1-t^2)(1-t^4)}
\times \bigl(1+3t^2+4t^4+3t^6+t^8\bigr)
+ \bigl(t^{2(g-1)}+t^{2g}\bigr)(1+t)^{4(g-1)}
$}
\]
\end{theorem}

\nocite{*}
\bibliography{biblio}
\bibliographystyle{amsalpha}
\end{document}